\documentclass[11pt]{amsart}
\usepackage{epsfig, graphics, psfrag, color}
\usepackage{amscd} 

\graphicspath{{figures/}}

\let\oldmarginpar\marginpar
\renewcommand\marginpar[1]{\oldmarginpar[\raggedleft\footnotesize #1]%
{\raggedright\footnotesize #1}}

\newcommand{\nin}{{\notin}}
\newcommand{\D}{{\partial}} 
\newcommand{\bdy}{{\partial}} 

\renewcommand{\phi}{{\varphi}}

\newcommand{\HH}{{\mathbb{H}}}
\newcommand{\RR}{{\mathbb{R}}}
\newcommand{\ZZ}{{\mathbb{Z}}}

\newcommand{\CC}{{\mathbb{C}}}
\newcommand{\PSL}{{\mathrm{PSL}}}

\def\co{\colon\thinspace}

\newcommand{\F}{{\mathcal{F}}}

\theoremstyle{plain}
\newtheorem{theorem}{Theorem}[section]

\newtheorem{cor}[theorem]{Corollary}
\newtheorem{lemma}[theorem]{Lemma}
\newtheorem{prop}[theorem]{Proposition}

\newtheorem{conjecture}[theorem]{Conjecture}
\newtheorem{question}[theorem]{Question}

\newtheorem*{namedtheorem}{\theoremname}
\newcommand{\theoremname}{testing}

\theoremstyle{definition}
\newtheorem{define}[theorem]{Definition}
\newtheorem{algorithm}[theorem]{Algorithm}
\newtheorem{procedure}[theorem]{Procedure}
\newtheorem{example}[theorem]{Example}
\newtheorem{remark}[theorem]{Remark}

\begin{document}


\title{Geodesics and compression bodies}

\author{Marc Lackenby}

\address{Mathematical Institute, University of Oxford, Oxford, UK}

\author{Jessica S. Purcell}

\address{Department of Mathematics, Brigham Young University, Provo,
  UT 84602, USA}


\begin{abstract}
We consider hyperbolic structures on the compression body $C$ with
genus 2 positive boundary and genus 1 negative boundary.  Note that
$C$ deformation retracts to the union of the torus boundary and a
single arc with its endpoints on the torus.  We call this arc the core
tunnel of $C$.  We conjecture that, in any geometrically finite
structure on $C$, the core tunnel is isotopic to a geodesic.  By
considering Ford domains, we show this conjecture holds for many
geometrically finite structures.  Additionally, we give an algorithm
to compute the Ford domain of such a manifold, and a procedure
which has been implemented to visualize many of these Ford domains.
Our computer implementation gives further evidence for the conjecture.
\end{abstract}

\maketitle

\newcommand{\mat}[2][cccc]{\left(\begin{array}{#1} #2\\
	\end{array}\right)}

\section{Introduction}\label{sec:intro}
For a hyperbolic manifold $M$ with torus boundary component $\partial
M_0$, every homotopically nontrivial arc in $M$ with endpoints on
$\partial M_0$ is homotopic to a geodesic.  However, it seems to be a
difficult problem to identify arcs in $M$ which are isotopic to a
geodesic, given only a topological description of $M$.

One place this problem arises is in the study of unknotting tunnels.
An \emph{unknotting tunnel} for a 3--manifold $M$ with torus boundary
components is defined to be an arc $\tau$ from $\partial M$ to
$\partial M$ such that $M\setminus N(\tau)$ is a handlebody.
Manifolds (other than a solid torus) that admit unknotting tunnels are
\emph{tunnel number one} manifolds.  Adams asked whether the
unknotting tunnel of a hyperbolic tunnel number one manifold is always
isotopic to a geodesic \cite{adams:tunnels}.  This has been shown to
be the case for many classes of hyperbolic tunnel number one manifolds
(\cite{adams-reid}, \cite{sakuma-weeks}).  Recently, Cooper, Futer,
and Purcell showed that the conjecture is true for a generic manifold,
in an appropriate sense of generic \cite{cfp:tunnels}.  The original
question still remains open, however.

The purpose of this paper is to present and motivate a related
question.  Any tunnel number one manifold is built by attaching a
compression body $C$ to a handlebody, and the unknotting tunnel
corresponds to an arc $\tau$ in the compression body.  We call $\tau$
the \emph{core tunnel} of $C$.  Given Adams' question on whether an
unknotting tunnel is isotopic to a geodesic, it seems natural to ask
whether the arc $\tau$ is isotopic to a geodesic under a complete
hyperbolic structure on $C$.

The compression body $C$ admits many complete hyperbolic structures.
Here, we examine those that are geometrically finite, and show that
for many such structures, the core tunnel is isotopic to a geodesic.
In order to investigate such structures, we develop algorithms to find
the Ford domains for geometrically finite structures on $C$.  We
present one algorithm that is guaranteed to find the Ford domain in
finite time and terminate, but which is impractical in practice, and a
procedure which has been implemented for the computer, which will find
the Ford domain and terminate for large families of geometrically
finite structures, and which we conjecture will always find the Ford
domain.

Computer investigation and the theorems proven for families of
geometrically finite hyperbolic structures lead us to the following
conjecture.

\begin{conjecture}
Let $C$ be a compression body with $\partial_{-}C$ a torus, and
$\partial_{+}C$ a genus two surface.  Suppose $C$ is given a
geometrically finite hyperbolic structure.  Then the core tunnel of
$C$ is isotopic to a geodesic.
\label{conj:core-tunnel}
\end{conjecture}

In fact, we conjecture something stronger.  We conjecture that the
core tunnel is not only isotopic to a geodesic, but always dual to a
face of the Ford domain.  This is Conjecture \ref{conj:dual-ford},
explained in Section \ref{sec:geodesics}.

The techniques of this paper can be seen as an extension of work of
J{\o}rgensen \cite{jorgensen}, who found Ford domains of geometrically
finite structures on $S\times \RR$, where $S$ is a once--punctured
torus.  J{\o}rgensen's work was extended and expanded by others,
including Akiyoshi, Sakuma, Wada, and Yamashita \cite{aswy,
  aswy:book}.  Wada implemented an algorithm to determine Ford domains
of these manifolds \cite{wada:opti}.

A complete understanding of the geometry of compression bodies, for
example through a study of Ford domains, could lead to many
interesting applications, since compression bodies are building blocks
of more complicated manifolds via Heegaard splitting techniques.
With Cooper, we have already applied some of the ideas in this paper
to build tunnel number one manifolds with arbitrarily long
unknotting tunnels \cite{clp:length}.

\subsection{Acknowledgements}
Both authors were supported by the Leverhulme trust.  Lackenby was
supported by an EPSRC Advanced Research Fellowship.  Purcell was
supported by NSF grants and the Alfred P.~Sloan foundation.

\section{Background and preliminary material}\label{sec:prelim}
In this section we review terminology and results used throughout the
paper.  Our intent is to make this paper as self--contained as
possible, and also to emphasize relations between the geometry and
topology of compression bodies.  

First, we review definitions and results on compression bodies, which
are the manifolds we study.  Next, we review what it means for these
manifolds to admit a geometrically finite hyperbolic structure.  We
then recall the definition of a Ford domain, since we will be using
Ford domains to examine geometrically finite hyperbolic structures on
compression bodies.  We also give a few definitions relevant to Ford
domains, such as visible isometric spheres, Ford spines, and complexes
dual to Ford spines.  Ford domains of geometrically finite manifolds
are finite sided polyhedra; thus we can often identify a Ford domain
using the Poincar{\'e} polyhedron theorem.  Finally, we review this
theorem and some of its relevant consequences.

\subsection{Compression bodies}

The manifolds we study in this paper are compression bodies with
negative boundary a single torus, and positive boundary a genus $2$
surface.

Recall that a \emph{compression body} $C$ is either a handlebody, or
the result of taking the product $S\times I$ of a closed, oriented
(possibly disconnected) surface $S$ and the interval $I=[0,1]$, and
attaching 1--handles to $S\times\{1\}$ such that the result is
connected.  The \emph{negative boundary} is $S\times\{0\}$ and is
denoted $\partial_{-}C$.  When $C$ is a handlebody, $\partial_{-}C
= \emptyset$.  The \emph{positive boundary} is $\partial
C \setminus \partial_{-}C$, and is denoted $\partial_{+}C$.

Let $C$ be the compression body for which $\partial_{-}C$ is a torus
and $\partial_{+}C$ is a genus $2$ surface.  We will call this the
\emph{$(1;2)$--compression body}, where the numbers $(1;2)$ refer to
the genus of the boundary components.  Note the $(1;2)$--compression
body is formed by taking a torus $T^2$ crossed with $[0,1]$ and
attaching a single 1--handle to $T^2 \times \{1\}$.  The 1--handle
retracts to a single arc, the core of the 1--handle.

Let $\tau$ be the union of the core of the 1--handle with two vertical
arcs in $S \times [0,1]$ attached to its endpoints.  Thus, $\tau$ is a
properly embedded arc in $C$, and $C$ is a regular neighborhood of
$\partial_- C \cup \tau$.  We refer to $\tau$ as the \emph{core
  tunnel} of $C$.  See Figure \ref{fig:comp-body}, which first
appeared in \cite{clp:length}.

\begin{figure}
	\centerline{\includegraphics[width=3in]{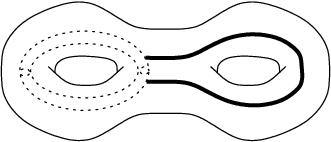}}
\caption{The $(1;2)$--compression body. The core tunnel is the thick
	line shown, with endpoints on the torus boundary.}
\label{fig:comp-body}
\end{figure}

The fundamental group of a $(1;2)$--compression body $C$ is isomorphic
to $(\ZZ\times\ZZ)\ast \ZZ$.  We will denote the generators of the
$\ZZ\times\ZZ$ factor by $\alpha$, $\beta$, and we will denote the
generator of the second factor by $\gamma$.

\subsection{Hyperbolic structures}

We are interested in the isotopy class of the arc $\tau$ when we put a
complete hyperbolic structure on the interior of the
$(1;2)$--compression body $C$.  We obtain such a structure by taking a
discrete, faithful representation $\rho\co \pi_1(C) \to {\rm
  PSL}(2,\CC)$ and considering the manifold $\HH^3/\rho(\pi_1(C))$.

\begin{define}
A discrete subgroup $\Gamma < {\rm PSL}(2,\CC)$ is \emph{geometrically
finite} if $\HH^3/\Gamma$ admits a finite--sided, convex fundamental
domain.  In this case, we will also say that the manifold
$\HH^3/\Gamma$ is \emph{geometrically finite}.
\label{def:gf-group}
\end{define}

The following gives a useful fact about geometrically finite groups in
$\PSL(2,\CC)$.

\begin{theorem}[Bowditch, Proposition 5.7 \cite{bowditch}]
If a subgroup $\Gamma < {\rm PSL}(2,\CC)$ is geometrically finite, then
every convex fundamental domain for $\HH^3/\Gamma$ has finitely many
faces.
\label{thm:bowditch}
\end{theorem}

\begin{define}
A discrete subgroup $\Gamma < \PSL(2,\CC)$ is \emph{minimally
parabolic} if it has no rank one parabolic subgroups.
\label{def:min-parabolic}
\end{define}

Thus for a discrete, faithful representation $\rho \co \pi_1(M) \to
\PSL(2,\CC)$, the image $\rho(\pi_1(M))$ will be minimally parabolic
if for all $g \in \pi_1(C)$, the element $\rho(g)$ is parabolic if and
only if $g$ is conjugate to an element of the fundamental group of a
torus boundary component of $M$.

\begin{define}
  A discrete, faithful representation $\rho\co\pi_1(M)\to \PSL(2,\CC)$
  is a \emph{minimally parabolic geometrically finite uniformization
    of $M$} if $\rho(\pi_1(M))$ is minimally parabolic and
  geometrically finite, and $\HH^3/\rho(\pi_1(M))$ is homeomorphic to
  the interior of $M$.
\label{def:gf}
\end{define}

\subsection{Isometric spheres and {F}ord domains}

To examine structures on $C$, we examine paths of Ford domains.  This
is similar to the technique of J{\o}rgensen \cite{jorgensen},
developed and expanded by Akiyoshi, Sakuma, Wada, and Yamashita
\cite{aswy:book}, to study hyperbolic structures on punctured torus
bundles.  Much of the basic material on Ford domains which we review
here can also be found in \cite{aswy:book}.

Throughout this subsection, let $M=\HH^3/\Gamma$ be a hyperbolic
manifold with a single rank 2 cusp, for example, the
$(1;2)$--compression body.  In the upper half space model for $\HH^3$,
assume the point at infinity in $\HH^3$ projects to the cusp.  Let $H$
be any horosphere about infinity.  Let $\Gamma_\infty < \Gamma$ denote
the subgroup that fixes $H$.  By assumption, $\Gamma_\infty \cong
\ZZ\times\ZZ$.

\begin{define}
For any $g \in \Gamma \setminus\Gamma_\infty$, $g^{-1}(H)$ will be a
horosphere centered at a point of $\CC$, where we view the boundary at
infinity of $\HH^3$ to be $\CC \cup \{\infty\}$.  Define the set
$I(g)$ to be the set of points in $\HH^3$ equidistant from $H$ and
$g^{-1}(H)$.  Then $I(g)$ is the \emph{isometric sphere} of $g$.
\label{def:isometric-sphere}
\end{define}

Note that $I(g)$ is well--defined even if $H$ and $g^{-1}(H)$ overlap.
It will be a Euclidean hemisphere orthogonal to the boundary $\CC$ of
$\HH^3$.

The following is well known, and follows from standard calculations.
We include a proof for completeness.

\begin{lemma}
If $$g=\mat{a&b\\c&d}\in {\rm PSL}(2,\CC),$$ then the center of the
Euclidean hemisphere $I({g^{-1}})$ is $g(\infty) = a/c$.  Its
Euclidean radius is $1/|c|$.
\label{lemma:iso-center-rad}
\end{lemma}

\begin{proof}
The fact that the center is $g(\infty)=a/c$ is clear.

Consider the geodesic running from $\infty$ to $g(\infty)$.  It
consists of points of the form $(a/c, t)$ in $\CC \times \RR^+ \cong
\HH^3$.  It will meet the horosphere $H$ about infinity at some height
$t=h_1$, and the horosphere $g(H)$ at some height $t=h_0$.  The radius
of the isometric sphere $I({g^{-1}})$ is the height of the point
equidistant from points $(a/c, h_0)$ and $(a/c, h_1)$.

Note that $g^{-1}( g(H)) = H$, and hence $h_1$ is given by the height
of $g^{-1}(a/c, h_0)$, which can be computed to be $(-d/c,
1/(|c|^2h_0))$.  Thus $h_1 =1/(|c|^2h_0)$.  Then the point equidistant
from $(a/c, h_0)$ and $(a/c, 1/(|c|^2h_0))$ is the point of height $h
= 1/|c|$.
\end{proof}

\begin{define}
  Let $B(g)$ denote the \emph{open} half ball bounded by $I(g)$, and
  define $\F$ to be the set
  $$\F = \HH^3 \setminus \bigcup_{g \in \Gamma \setminus \Gamma_\infty} B(g).$$ 
  Note $\F$ is invariant under $\Gamma_\infty$, which acts by
  Euclidean translations on $\HH^3$.  We call $\F$ the
  \emph{equivariant Ford domain}.
  \label{def:F}
\end{define}

When $H$ bounds a horoball $H_\infty$ that projects to an embedded
horoball neighborhood about the rank 2 cusp of $M$, $\F$ is the set of
points in $\HH^3$ which are at least as close to $H_\infty$ as to any
of its translates under $\Gamma \setminus \Gamma_\infty$.  Provided
$\Gamma$ is discrete, such an embedded horoball neighborhood of the
cusp always exists, by the Margulis lemma.

\begin{define}
A \emph{vertical fundamental domain for $\Gamma_\infty$} is a
fundamental domain for the action of $\Gamma_\infty$ cut out by
finitely many vertical geodesic planes in $\HH^3$.
\end{define}

\begin{define}
A \emph{Ford domain} of $M$ is the intersection of $\F$ with a
vertical fundamental domain for the action of $\Gamma_\infty$.
\label{def:ford-domain}
\end{define}

A Ford domain is not canonical because the choice of fundamental
domain for $\Gamma_\infty$ is not canonical.  However, the equivariant
Ford domain $\F$ in $\HH^3$ is canonical, and for purposes of this
paper, $\F$ is often more useful than the actual Ford domain.

Note that Ford domains are convex fundamental domains (\emph{cf.}
\cite[Proposition A.1.2]{aswy:book}).  Thus we have the following
corollary of Bowditch's Theorem \ref{thm:bowditch}.

\begin{cor}
$M=\HH^3/\Gamma$ is geometrically finite if and only if a Ford domain
  for $M$ has a finite number of faces.
\label{cor:bowditch}
\end{cor}

\begin{example}
Let $c\in \CC$ be any complex number such that $|c|>2$, and let $a$
and $b$ in $\CC$ be linearly independent over $\RR$ with $|a|>2|c|$,
$|b|>2|c|$.  Let $\rho\co\pi_1(C) \to \PSL(2,\CC)$ be the
representation defined by
$$\rho(\alpha) = \mat{1&a\\0&1}, \quad \rho(\beta) =
\mat{1&b\\0&1}, \quad \rho(\gamma)=\mat{c&-1\\1&0}.$$
(Recall that $\alpha$ and $\beta$ denote the generators of the $\ZZ
\times \ZZ$ factor of $\pi_1(C)$, and $\gamma$
denotes an additional generator of $\pi_1(C)$.)

By Lemma \ref{lemma:iso-center-rad}, $I({\rho(\gamma)})$ has center
$0$, radius $1$, and $I({\rho(\gamma^{-1})})$ has center $c\in\CC$,
radius $1$.  Since $|c|>2$, $I({\rho(\gamma)})$ will not meet
$I({\rho(\gamma^{-1})})$.  By choice of $\rho(\alpha)$, $\rho(\beta)$,
all translates of $I({\rho(\gamma)})$ and $I({\rho(\gamma^{-1})})$ under
$\Gamma_\infty$ are disjoint.

We will see in Lemma \ref{lemma:simple-ford} that $\rho$ gives a
minimally parabolic geometrically finite uniformization of $C$, and
that for this example, $\F$ consists of the exterior of (open)
half--spaces $B(\rho(\gamma))$ and $B({\rho(\gamma^{-1})})$, bounded
by $I(\rho(\gamma))$ and $I(\rho({\gamma^{-1}}))$, respectively, as
well as translates of these two isometric spheres under
$\Gamma_\infty$.  Thus we will show that the Ford domain for this
example is as shown in Figure \ref{fig:simple-ford}.  Before proving
this fact, we need additional definitions and lemmas.  We use this
example to illustrate these definitions and lemmas.
\label{ex:simple-ford}
\end{example}

\begin{figure}
\begin{center}
\begin{tabular}{ccc}
	\makebox{\begin{tabular}{c}
				\input{figures/simple-forddomain.pstex_t} \\
			\end{tabular}	} & 
	\hspace{1in} &
	\makebox{\begin{tabular}{c}
			\vspace{-0.2in} \\
			\includegraphics[width=1.5in]{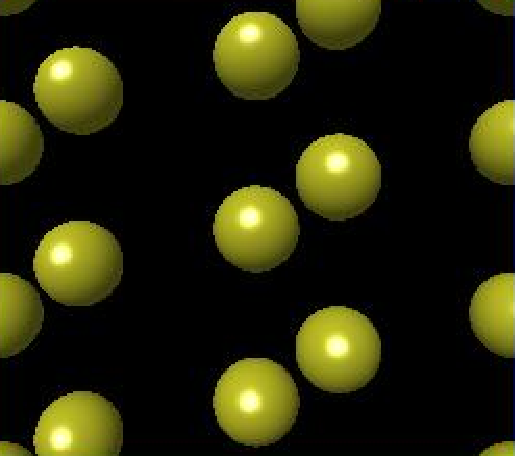}
			\end{tabular}}
\end{tabular}
\end{center}
\caption{Left: Schematic picture of the Ford domain of Example
  \ref{ex:simple-ford}.  Right: Three dimensional view of $\F$ in
  $\HH^3$, for $c=2+i$, $a=6+2i$, and $b=4.5i$.}
\label{fig:simple-ford}
\end{figure}

\subsection{Visible faces and {F}ord domains}

Let $M = \HH^3/\Gamma$ be a hyperbolic manifold with a single rank two
cusp, and let $\Gamma_\infty < \Gamma$ denote a maximal rank two
parabolic subgroup, which we may assume fixes the point at infinity in
$\HH^3$.  Notice that $\F$, the equivariant Ford domain of $M$, has a
natural cell structure.

\begin{define}
  Let $g \in \Gamma \setminus \Gamma_\infty$.  We say $I(g)$ is
  \emph{visible} if there exists a 2--dimensional cell of the cell
  structure on $\F$ contained in $I(g)$.

  Similarly, we say the intersection of isometric spheres $I(g_1) \cap
  \dots \cap I(g_n)$ is \emph{visible} if there exists a cell of $\F$
  contained in $I(g_1) \cap \dots \cap I(g_n)$ of the same dimension
  as $I(g_1) \cap \dots \cap I(g_n)$.
\label{def:visible}
\end{define}

Thus in Example \ref{ex:simple-ford}, we claim that the only visible
isometric spheres are $I(\rho(\gamma))$, $I(\rho({\gamma^{-1}}))$, and
the translates of these under $\Gamma_\infty$.  There are no visible
edges for this example.

There is an alternate definition of visible, Lemma
\ref{lemma:alt-visible}.  Let $H$ be a horosphere about infinity that
bounds a horoball which is embedded under the projection to $M$.

\begin{lemma}
For $g \in \Gamma \setminus \Gamma_\infty$, $I(g)$ is visible if and
only if there exists an open set $U \subset \HH^3$ such that $U\cap
I(g)$ is not empty, and for every $x \in U\cap I(g)$ and every $h \in
\Gamma \setminus \Gamma_\infty$, the hyperbolic distances satisfy
$$d(x,h^{-1}(H)) \geq d(x, H) = d(x, g^{-1}H).$$

Similarly, if $I(g) \cap I(h)$ is not empty, then it is visible if and
only if there exists an open $U \subset \HH^3$ such that $U \cap I(g)
\cap I(h)$ is not empty, and for every $x\in U\cap I(g) \cap I(h)$ and
every $k \in \Gamma \setminus \Gamma_\infty$,
$$d(x, k^{-1}H) \geq d(x,H) = d(x, g^{-1}H) = d(x, h^{-1}H).$$
\label{lemma:alt-visible}
\end{lemma}

\begin{proof}
An isometric sphere, or interesction of isometric spheres, is visible
if and only if it contains a cell of $\F$ of the same dimension.  This
will happen if and only if there is some open set $U$ in $\HH^3$ which
intersects the isometric sphere, or intersections of isometric
spheres, in the cell of $\F$ in $\HH^3$.  The result follows now by
definition of $\F$: a point $x$ is in $\F$ if and only if it is not
contained in any open half space $B(k)$, $k \in \Gamma \setminus
\Gamma_\infty$, if and only if $d(x, H) \leq d(x, k^{-1}H)$.
\end{proof}

We can say something even stronger for isometric spheres:

\begin{lemma}
For $\Gamma$ discrete, the following are equivalent.
\begin{enumerate}
\item\label{item:1vis} The isometric sphere $I(g)$ is visible.
\item\label{item:2vis} There exists an open set $U \subset \HH^3$ such
  that $U \cap I(g)$ is not empty and for any $x \in U \cap I(g)$ and
  any $h\in \Gamma\setminus (\Gamma_\infty \cup \Gamma_\infty g)$,
$$d(x, h^{-1}H) > d(x,H) = d(x,g^{-1}H).$$
\item\label{item:3vis} $I(g)$ is not contained in  $\bigcup_{h \in \Gamma
  \setminus(\Gamma_\infty \cup \Gamma_\infty g)} \overline{B(h)}$.
\end{enumerate}
\label{lemma:visible-strict-inequality}
\end{lemma}

\begin{proof}
If \eqref{item:2vis} holds, then Lemma \ref{lemma:alt-visible} implies
$I(g)$ is visible.  Conversely, suppose $I(g)$ is visible.  Let $U$ be
as in Lemma \ref{lemma:alt-visible}, so that for all $x \in U\cap
I(g)$, and all $h\in \Gamma \setminus \Gamma_\infty$, $d(x,h^{-1}H)
\geq d(x,H) = d(x, g^{-1}H)$.  Suppose there is some $h \in \Gamma
\setminus \Gamma_\infty$ such that for all $x \in U \cap I(g)$ we have
equality: $d(x, h^{-1}H) = d(x,H) = d(x, g^{-1}H)$.  Then the
isometric spheres $I(h)$ and $I(g)$ must agree on an open subset, hence
they must agree everywhere.  In particular, their centers must agree:
$g^{-1}(\infty) = h^{-1}(\infty)$.

Now, notice that $g^{-1} \Gamma_\infty g$ is the subgroup of $\Gamma$
fixing $g^{-1}(\infty)$, since $\alpha$ fixes $g^{-1}(\infty)$ if and
only if $g\alpha g^{-1}$ fixes infinity, so lies in $\Gamma_\infty$.
Next note that since $I(g) = I(h)$, $g^{-1}h$ fixes $g^{-1}(\infty)$.
So $g^{-1}h \in g^{-1}\Gamma_\infty g$.  Thus $h\in \Gamma_\infty g$.
We have shown \eqref{item:1vis} if and only if \eqref{item:2vis}.  

Finally, \eqref{item:2vis} clearly implies \eqref{item:3vis}.  If
$I(g)$ is not visible, then for any $x \in I(g)$, either $x \nin \F$,
which implies $x \in \bigcup_{h \in \Gamma \setminus(\Gamma_\infty
  \cup \Gamma_\infty g)} \overline{B(h)}$, or $x$ is in a cell of $\F$
with dimension at most $1$.  In this case, $x \in I(h)$ for some $h
\in \Gamma \setminus (\Gamma_\infty \cup \Gamma_\infty g)$.  Thus
\eqref{item:3vis} implies \eqref{item:1vis}.
\end{proof}

Notice that in the above proof, we showed that if two isometric
spheres $I(g)$ and $I(h)$ agree, then $h \in \Gamma_\infty g$.  It is
clear that if $h \in \Gamma_\infty g$, then $I(g) = I(h)$.

We now present two results on visible faces of the Ford domain.
Again these are well known, but we include proofs for completeness.

\begin{lemma}
Let $\Gamma$ be a discrete, torsion free subgroup of $\PSL(2,\CC)$
with a rank two parabolic subgroup $\Gamma_\infty$ fixing the point at
infinity, and let $g\in \Gamma \setminus \Gamma_\infty$.  Then $I(g)$
is visible if and only if $I({g^{-1}})$ is visible.  Moreover, $g$
takes $I(g)$ isometrically to $I({g^{-1}})$, sending the half space
$B(g)$ bounded by $I(g)$ to the exterior of the half space
$B(g^{-1})$.
\label{lemma:g-ginv}
\end{lemma}

\begin{proof}
Let $H$ be a horosphere about infinity in $\HH^3$ that bounds a
horoball which projects to an embedded neighborhood of the cusp of
$M$.

First, note that under $g$, $I(g)$ is mapped isometrically to
$I({g^{-1}})$, since $g$ takes $H$ to $g(H)$, and $g^{-1}(H)$ to $H$,
and hence takes $I(g)$ to the set of points equidistant from these two
horospheres.  This is the isometric sphere $I({g^{-1}})$.  Note the
half space $B(g)$, which contains $g^{-1}(H)$, must be mapped to the
exterior of $B({g^{-1}})$, which contains $H$, as claimed.

Suppose $I(g)$ is visible.  Then there exists an open set $U \subset
\HH^3$, with $U \cap I(g)$ not empty, so that for every $x$ in $I(g)
\cap U$, and for every $h \in \Gamma \setminus \Gamma_\infty$, $d(x,
h^{-1}(H)) \geq d(x, H)= d(x, g^{-1}(H))$.

Now consider the action of $g$ on this picture.  The set $g(U)$ is
open in $\HH^3$, and for all $y \in g(U) \cap I(g^{-1})$, we have $y
= g(x)$, for some $x \in U \cap I(g)$, so the distance $d(y, H) =
d(g(x), g g^{-1}(H)) \leq d(g(x), g h^{-1}(H)) = d(y, g h^{-1}(H))$,
for all $h \in \Gamma \setminus \Gamma_\infty$.  So $I({g^{-1}})$ is
visible.

To finish, apply the same proof to $g^{-1}$.
\end{proof}

\begin{lemma}
Gluing isometric spheres corresponding to $\rho(\gamma)$ and
$\rho(\gamma^{-1})$ of Example \ref{ex:simple-ford} gives a manifold
homeomorphic to the interior of the $(1;2)$--compression body $C$.
\label{lemma:simple-ford-uniform}
\end{lemma}

\begin{proof}
In the example, first glue sides of the vertical fundamental domain
via the parabolic transformations fixing infinity.  The result is
homeomorphic to the cross product of a torus and an open interval
$(0,1)$.  Next glue the face $I(\rho(\gamma))$ to
$I(\rho(\gamma^{-1}))$ via $\gamma$.  The result is
topologically equivalent to attaching a 1--handle, yielding a manifold
homeomorphic to $C$.
\end{proof}

\begin{lemma}
Let $\Gamma$ be a discrete, torsion free subgroup of $\PSL(2,\CC)$
with a rank two parabolic subgroup $\Gamma_\infty$ fixing the point at
infinity.  Suppose $g, h \in \Gamma\setminus\Gamma_\infty$, with
$I(g)$ and $I(h)$ visible, and suppose $I(g) \cap I(h)$ is visible.
Then $I(gh^{-1})\cap I(h^{-1})$ is visible, and $h$ maps the visible
portion of $I(g) \cap I(h)$ isometrically to the visible portion of
$I(gh^{-1}) \cap I(h^{-1})$.  In addition, there must be some visible
isometric sphere $I(k)$, not equal to $I(h^{-1})$, such that $I(k)
\cap I(h^{-1}) = I(gh^{-1}) \cap I(h^{-1})$.
\label{lemma:edge-visible}
\end{lemma}

Notice that in Lemma \ref{lemma:edge-visible}, $I(k)$ may be equal to
$I(gh^{-1})$, but is not necessarily so.  In fact, $I(gh^{-1})$ may
not be visible, such as in the case that there is a quadrilateral dual
to $I(g) \cap I(h)$.  We discuss dual faces later.

\begin{proof}
Let $H$ be a horosphere about infinity which bounds a horoball that
projects to an embedded neighborhood of the cusp of $M$.  Suppose
$I(g) \cap I(h)$ is visible.  By Lemma \ref{lemma:alt-visible}, there
exists an open set $U \subset \HH^3$ such that for all $x \in U \cap
(I(g) \cap I(h))$, and all $k \in \Gamma \setminus \Gamma_\infty$, the
hyperbolic distance $d(x,H)$ is less than or equal to the hyperbolic
distance $d(x, k^{-1}(H))$.  Since $x \in I(g) \cap I(h)$, we also
have $d(x, g^{-1}H) = d(x, h^{-1}H) = d(x, H)$.

Apply $h$ to this picture.  We obtain:
$$d(h(x), hg^{-1}H) = d(h(x), H) =d(h(x), h H) \leq d(h(x),
hk^{-1}H)$$ for all $k \in \Gamma \setminus \Gamma_\infty$.  Thus for
all $y$ in the intersection of the open set $h(U)$ and $I(gh^{-1})
\cap I(h^{-1})$, $y = h(x)$ satisfies the inequality of Lemma
\ref{lemma:alt-visible}, and so $I(gh^{-1}) \cap I(h^{-1})$ is
visible.  Since this works for any such open set $U$, and the 1--cell
of $\F$ contained in $I(g) \cap I(h)$ may be covered with such open
sets, $h$ maps visible portions isometrically.

Finally, since $I(gh^{-1}) \cap I(h^{-1})$ is visible, it contains a
1--dimensional cell of $\F$.  There must be two 2--dimensional cells
of $\F$ bordering $I(gh^{-1}) \cap I(h^{-1})$.  One of these is
contained in $I(h^{-1})$, using the fact that $I(h)$ is visible and
Lemma \ref{lemma:g-ginv}.  The other must be contained in some $I(k)$
(possibly, but not necessarily $I({gh^{-1}})$), and so this $I(k)$ is
visible.
\end{proof}

The first part of Lemma \ref{lemma:edge-visible} is a portion of what
Akiyoshi, Sakuma, Wada, and Yamashita call the \emph{chain rule for
  isometric circles} \cite[Lemma 4.1.2]{aswy:book}.

Additionally, we present a result that allows us to identify
geometrically finite uniformizations that are minimally parabolic.

\begin{lemma}
Suppose $\rho\co \pi_1(C) \to {\rm PSL}(2,\CC)$ is a geometrically
finite uniformization.  Suppose none of the visible isometric spheres
of the Ford domain of $\HH^3/\rho(\pi_1(C))$ are visibly tangent on
their boundaries.  Then $\rho(\pi_1(C))$ is minimally parabolic.
\label{lemma:min-parabolic}
\end{lemma}

By \emph{visibly} tangent, we mean the following.  Set $\Gamma =
\rho(\pi_1(C))$, and assume a neighborhood of infinity in $\HH^3$
projects to the rank two cusp of $\HH^3/\Gamma$, with $\Gamma_\infty <
\Gamma$ fixing infinity in $\HH^3$.  For any $g \in \Gamma \setminus
\Gamma_\infty$, the isometric sphere $I(g)$ has boundary that is a
circle on the boundary $\CC$ at infinity of $\HH^3$.  This circle
bounds an open disk $D(g)$ in $\CC$.  Two isometric spheres $I(g)$ and
$I(h)$ are \emph{visibly tangent} if their corresponding disks $D(g)$
and $D(h)$ are tangent on $\CC$, and for any other $k \in \Gamma
\setminus \Gamma_\infty$, the point of tangency is not contained in
the open disk $D(k)$.

\begin{proof}
Suppose $\rho(\pi_1(C))$ is not minimally parabolic.  Then it must
have a rank 1 cusp.  Apply an isometry to $\HH^3$ so that the point at
infinity projects to this rank 1 cusp.  The Ford domain becomes a
region $P$ meeting this cusp, with finitely many faces.  Take a
horosphere about infinity sufficiently small that the intersection of
the horosphere with $P$ gives a subset of Euclidean space with sides
identified by elements of $\rho(\pi_1(C))$, conjugated appropriately.

The side identifications of this subset of Euclidean space, given by
the side identifications of $P$, generate the fundamental group of the
cusp.  But this is a rank 1 cusp, hence its fundamental group is
$\ZZ$.  Therefore, the side identification is given by a single
Euclidean translation.  The Ford domain $P$ intersects this horosphere
in an infinite strip, and the side identification glues the strip into
an annulus.  Note this implies two faces of $P$ are tangent at
infinity.

Now apply an isometry, taking us back to our usual view of $\HH^3$,
with the point at infinity projecting to the rank 2 cusp of the
$(1;2)$--compression body $\HH^3/\rho(\pi_1(C))$.  The two faces of
$P$ tangent at infinity are taken to two isometric spheres of the Ford
domain, tangent at a visible point on the boundary at infinity.
\end{proof}

We will see that the converse to Lemma \ref{lemma:min-parabolic} is
not true.  There exist examples of geometrically finite
representations for which two visible isometric spheres are visibly
tangent, and yet the representation is still minimally parabolic.
Such an example is given, for example, in Example \ref{ex:bumping},
with $t=\sqrt{3}$.

\begin{remark}
In Example \ref{ex:simple-ford}, we claimed that the only visible
isometric spheres are those of $I(\rho(\gamma))$,
$I(\rho(\gamma^{-1}))$, and their translates under $\Gamma_\infty$.
Since none of these isometric spheres are visibly tangent, provided
the claim is true, Lemma \ref{lemma:min-parabolic} will imply that
this representation is minimally parabolic.
\end{remark}

\subsection{The {F}ord spine}

Let $\Gamma$ be discrete and geometrically finite.  When we glue the
Ford domain into the manifold $M=\HH^3/\Gamma$, the faces of the Ford
domain will be glued together in pairs to form $M$.

\begin{define}
The \emph{Ford spine} of $M$ is defined to be the image of the faces,
edges, and 0--cells of $\F$ under the covering $\HH^3\to M$.
\label{def:spine}
\end{define}

A spine usually refers to a subset of the manifold for which there is
a retraction of the manifold.  Using that definition, the Ford spine
is not strictly a spine.  However, the union of the Ford spine and the
non-toroidal boundary components will be a spine for a manifold $M$
with a single rank 2 cusp.  

To make that last sentence precise, recall that given a geometrically
finite uniformization $\rho$, the \emph{domain of discontinuity}
$\Omega$ is the complement of the limit set of $\rho(\pi_1(M))$ in the
boundary at infinity $\D_\infty \HH^3$.  See, for example, Marden
\cite[section 2.4]{marden-book}.

\begin{lemma}
Let $\rho$ be a minimally parabolic geometrically finite
uniformization of a 3--manifold $M$ with a single rank 2 cusp.  Then
the manifold $(\HH^3 \cup \Omega)/\rho(\pi_1(M))$ retracts onto the
union of the Ford spine and the boundary at infinity
$(\overline{\F} \cap \CC)/
\Gamma_\infty$.
\label{lemma:spine}
\end{lemma}

\begin{proof}
Let $H$ be a horosphere about infinity in $\HH^3$ that bounds a
horoball which projects to an embedded horoball neighborhood of the
cusp of $\HH^3/\rho(\pi_1(M))$.  Let $x$ be any point in $\F \cap
\HH^3$.  The nearest point on $H$ to $x$ lies on a vertical line
running from $x$ to infinity.  These vertical lines give a foliation
of $\F$.  All such lines have one endpoint on infinity, and the other
endpoint on $\overline{\F} \cap \CC$ or an isometric sphere of $\F$.
We obtain our retraction by mapping the point $x$ to the endpoint of
its associated vertical line, then quotienting out by the action of
$\rho(\pi_1(M))$.
\end{proof}

To any face $F_0$ of the Ford spine, we obtain an associated
collection of visible elements of $\Gamma$: those whose isometric
sphere projects to $F_0$ (or more carefully, a subset of their
isometric sphere projects to the face $F_0$).

\begin{define}
  We will say that an element $g$ of $\Gamma$ \emph{corresponds} to a
  face $F_0$ of the Ford spine of $M$ if $I(g)$ is visible and (the
  visible subset of) $I(g)$ projects to $F_0$.  In this case, we also
  say $F_0$ corresponds to $g$.  Notice the correspondence is not
  unique: if $g$ corresponds to $F_0$, then so does $g^{-1}$ and $w_0
  g^{\pm 1} w_1$ for any words $w_0, w_1 \in \Gamma_\infty$.
  \label{def:correspond}
\end{define}

\begin{remark}
Consider again the unifomization of $C$ given in Example
\ref{ex:simple-ford}.  We will see that the Ford domain of this
example has faces coming from a vertical fundamental domain and the
two isometric spheres $I({\rho(\gamma)})$ and
$I({\rho(\gamma^{-1})})$.  Hence the Ford spine of this manifold
consists of a single face, corresponding to $\rho(\gamma)$.
\end{remark}

\subsection{Poincar{\'e} polyhedron theorem}

We need a tool to identify the Ford domain of a hyperbolic manifold.
This tool will be Lemma \ref{lemma:finding-ford}.  The proof of that
lemma uses the Poincar{\'e} polyhedron theorem, which we use
repeatedly in this paper.  Those results we use most frequently are
presented in this subsection.  Our primary reference is Epstein and
Petronio \cite{epstein-petronio}, which contains a version of the
Poincar{\'e} theorem that does not require finite polyhedra.

The setup for the following theorems is the same.  We begin with a
finite number of elements of $\PSL(2,\CC)$, $g_1, g_2, \dots, g_n$, as
well as a parabolic subgroup $\Gamma_\infty \cong \ZZ \times \ZZ$ of
$\PSL(2,\CC)$, fixing the point at infinity.  Let $P$ be a polyhedron
cut out by isometric spheres corresponding to $\{g_1, \dots, g_n\}$
and $\{g_1^{-1}, \dots, g_n^{-1}\}$, as well as either:
\begin{enumerate}
\item all isometric spheres given by translations of $g_i$ and
  $g_i^{-1}$ under $\Gamma_\infty$, or
\item a vertical fundamental domain for the action of
  $\Gamma_\infty$. 
\end{enumerate}
An example of the former would be an equivariant Ford domain, $\F$.
An example of the latter would be a Ford domain.  Note that in both
cases, we allow $P$ to contain an open neighborhood of a point on the
boundary at infinity of $\HH^3$, so it will not necessarily have
finite volume.

Let $M$ be the object obtained from $P$ by gluing isometric spheres
corresponding to $g_j$ and $g_j^{-1}$ via the isometry $g_j$, for all
$j$, and then, if applicable, gluing faces of the vertical fundamental
domain by parabolic isometries in $\Gamma_\infty$.

\begin{theorem}[Poincar{\'e} polyhedron theorem, weaker version]
For $P$, $M$ as above, if $M$ is a smooth hyperbolic manifold, then
\begin{itemize}
\item the group $\Gamma$ generated by face pairings is discrete,
\item $\pi_1(M) \cong \Gamma$.
\end{itemize}
\label{thm:poincare2}
\end{theorem}

\begin{proof}
The result will follow essentially from \cite[Theorem
  5.5]{epstein-petronio}.  First we check the conditions of this
theorem.  Since $M$ is a smooth hyperbolic manifold, the condition
\emph{Pairing}, requiring faces to meet isometrically, will hold.
Similarly, the condition \emph{Cyclic} must hold, requiring the
monodromy around an edge in the identification to be the identity, and
sums of dihedral angles to be $2\pi$.  Condition \emph{Connected} is
automatically true for $P$ a single polyhedron (rather than a
collection of polyhedra).  Finally, note that since we have a finite
number of original isometric spheres corresponding to $g_1, \dots,
g_n$ and their inverses, and translation by an element in
$\Gamma_\infty$ moves an isometric sphere a fixed positive distance,
any isometric sphere of $P$ can meet only finitely many other
isometric spheres.  This is sufficient to imply condition
\emph{Locally finite}.  

We need to show the universal cover $\widetilde{M}$ of $M$ is
complete.  Since $M$ is a smooth hyperbolic manifold and $P$ is
complete, $M$ will be complete if and only if the link of its ideal
vertex inherits a Euclidean structure coming from horospherical cross
sections to $P$, by \cite[Theorem 3.4.23]{thurston:book}.  In the case
that $P$ is cut out only by isometric spheres and their translates
under $\Gamma_\infty$, there is nothing to show.  In the case that $P$
is cut out by a vertical fundamental domain, we know the holonomy of
the link of this vertex is given by the group $\Gamma_\infty$, which
is a rank 2 subgroup of $\PSL(2,\CC)$ fixing the point at infinity.
Thus it acts on a horosphere about infinity by Euclidean isometries,
and so $M$ is indeed complete.  It follows that $\widetilde{M}$ is
complete.

Thus all the conditions for \cite[Theorem 5.5]{epstein-petronio} hold,
and the developing map $\widetilde{M}\to \HH^3$ is a covering map,
with covering transformations generated by $\Gamma$.  It follows that
$\Gamma$ is discrete, and $\pi_1(M) \cong \Gamma$.
\end{proof}

\begin{theorem}[Poincar{\'e} polyhedron theorem]
For $P$, $M$ as above, and $\Gamma$ the group generated by face pairings,
suppose each face pairing maps a face of $P$ isometrically to another
face of $P$, and that for each edge $e$ of $M$, i.e.\ for each
equivalence class of intersections of isometric spheres under the
equivalence given by the gluing, the sum of dihedral angles about $e$
is $2\pi$, and the monodromy around the edge is the identity.  Then
\begin{itemize}
\item $M$ is a smooth hyperbolic manifold with $\pi_1(M) \cong \Gamma$, and
\item $\Gamma$ is discrete.
\end{itemize}
\label{thm:poincare1}
\end{theorem}

\begin{proof}
Again this follows from various results in \cite{epstein-petronio}.
Because faces of $P$ are mapped isometrically, we have the condition
\emph{Pairing}. The fact that dihedral angles sum to $2\pi$ and the
monodromy is the identity implies condition \emph{Cyclic}.  Again
because isometric spheres can meet only finitely many others in $P$,
we have condition \emph{Locally finite}, and because we have a single
polyhedron, we have condition \emph{Connected}.  When we send $P$ to
$\HH^3$ via the developing map, we may find a horosphere about
infinity disjoint from the isometric spheres forming faces of $P$.  In
the case that $P$ is cut out by a vertical fundamental domain, since
$\Gamma_\infty$ preserves this horosphere and acts on it by Euclidean
transformations, in the terminology of Epstein and Petronio, the
universal cover of the boundary of $M$ has a consistent horosphere.
This is true automatically if $P$ is not cut out by a vertical
fundamental domain.  Then by \cite[Theorem 6.3]{epstein-petronio}, the
universal cover $\widetilde{M}$ of $M$ is complete.  Now
Poincar{\'e}'s Theorem \cite[Theorem 5.5]{epstein-petronio} implies
the developing map $\widetilde{M} \to \HH^3$ is a covering map, hence
$M \cong \HH^3/\Gamma$ is a smooth, complete hyperbolic manifold with
$\pi_1(M) \cong \Gamma$ a discrete group.
\end{proof}

Our first application of Poincar{\'e}'s theorem is the following lemma,
which helps us identify Ford domains.

\begin{lemma}
Let $\Gamma$ be a subgroup of ${\rm PSL}(2,\CC)$ with a rank 2
parabolic subgroup $\Gamma_\infty$ fixing the point at infinity.

Suppose the isometric spheres corresponding to a finite set of
elements of $\Gamma$, as well as their translates under
$\Gamma_\infty$, cut out a region $\mathcal{G}$ so that the quotient
under face pairings and the group $\Gamma_\infty$ yields a smooth
hyperbolic manifold with fundamental group $\Gamma$.  Then $\Gamma$ is
discrete and geometrically finite, and $\mathcal{G}$ must be the
equivariant Ford domain of $\HH^3/\Gamma$.

Similarly, suppose the isometric spheres corresponding to a finite set
of elements of $\Gamma$, as well as a vertical fundamental domain for
$\Gamma_\infty$, cut out a polyhedron $P$, so that face pairings given
by the isometries corresponding to isometric spheres and to elements
of $\Gamma_\infty$ yield a smooth hyperbolic manifold with fundamental
group $\Gamma$.  Then $\Gamma$ is discrete and geometrically finite,
and $P$ must be a Ford domain of $\HH^3/\Gamma$.
\label{lemma:finding-ford}
\end{lemma}

\begin{proof}
In both cases, Theorem \ref{thm:poincare2} immediately implies that
$\Gamma$ is discrete.  The fact that $\Gamma$ is geometrically finite
follows directly from the definition.

In the case of the polyhedron $P$, suppose $P$ is not a Ford domain.
Since the Ford domain is only well--defined up to choice of
fundamental region for $\Gamma_\infty$, there is a Ford domain $F$
with the same choice of vertical fundamental domain for
$\Gamma_\infty$ as for $P$.  Since $P$ is not a Ford domain, $F$ and
$P$ do not coincide.  Because both are cut out by isometric spheres
corresponding to elements of $\Gamma$, there must be a visible face
that cuts out the domain $F$ that does not agree with any of those
that cut out the domain $P$.  Hence $F$ is a strict subset of $P$, and
there is some point $x$ in $\HH^3$ which lies in the interior of $P$,
but does not lie in the Ford domain.

Now consider the covering map $\phi\co \HH^3 \to \HH^3/\Gamma$.  This
map $\phi$ glues both $P$ and $F$ into the manifold $\HH^3/\Gamma$,
since both are fundamental regions for the manifold.  Now consider
$\phi$ applied to $x$.  Because $x$ lies in the interior of $P$, and
$P$ is a fundamental domain, there is no other point of $P$ mapped to
$\phi(x)$.  On the other hand, $x$ does not lie in the Ford domain
$F$.  Thus there is some preimage $y$ of $\phi(x)$ under $\phi$ which
does lie in $F$.  But $F$ is a subset of $P$.  Hence we have $y \neq
x$ in $P$ such that $\phi(x) = \phi(y)$.  This contradiction finishes
the proof in the case of the polyhedron $P$.

The proof for $\mathcal{G}$ is nearly identical.  Again if
$\mathcal{G}$ is not the equivariant Ford domain $\F$, then there is
an additional visible face of $\F$ besides those that cut out
$\mathcal{G}$, and again there is some point $x$ in $\HH^3$ which lies
in the interior of $\mathcal{G}$, but does not lie in $\F$.  Again the
covering map $\phi\co\HH^3 \to \HH^3/\Gamma$ glues $\mathcal{G}$ and
$\F$ into the manifold $\HH^3/\Gamma$, and again since a point $x$
lies in $\mathcal{G}$ but not in $\F$, we have some $y \neq x$ in $\F$
such that $\phi(x) = \phi(y)$.  Again this is a contradiction.
\end{proof}

We may now complete the proof that the Ford domain of the
representation of Example \ref{ex:simple-ford} is as shown in Figure
\ref{fig:simple-ford}.

\begin{lemma}
Let $\rho\co \pi_1(C) \to \PSL(2,\CC)$ be the representation given in
Example \ref{ex:simple-ford}.  Then $\rho$ gives a minimally parabolic
geometrically finite uniformization of $C$, and a Ford domain is given
by the intersection of a vertical fundamental domain for
$\Gamma_\infty$ with the half--spaces exterior to the two isometric
spheres $I({\rho(\gamma)})$ and $I({\rho(\gamma^{-1})})$.
\label{lemma:simple-ford}
\end{lemma}

\begin{proof}
We have seen that $I({\rho(\gamma)})$, $I({\rho(\gamma^{-1})})$, and
the translates of these isometric spheres under $\Gamma_\infty$ are
all disjoint.  Select a vertical fundamental domain for
$\Gamma_\infty$ which contains the isometric spheres
$I({\rho(\gamma)})$ and $I({\rho(\gamma^{-1})})$.  This is possible by
choice of $\rho(\alpha)$ and $\rho(\beta)$, particularly because the
translation lengths $|a|$ and $|b|$ are greater than $2|c|$.

Let $P$ be the region in the interior of the vertical fundamental
domain, exterior to the half--spaces $B(\rho(\gamma))$ and
$B({\rho(\gamma^{-1})})$ bounded by $I({\rho(\gamma)})$ and
$I({\rho(\gamma^{-1})})$, respectively.  Then when we identify
vertical sides of $P$ via elements of $\Gamma_\infty$, and identify
$I({\rho(\gamma)})$ and $I({\rho(\gamma^{-1})})$ via
$\rho(\gamma^{-1})$, the object we obtain is a smooth hyperbolic
manifold, by Theorem \ref{thm:poincare1}, since $P$ has no edges.
Lemma \ref{lemma:finding-ford} now implies that $P$ is a Ford domain
for $\HH^3/\Gamma$, and that $\Gamma$ is geometrically finite.  Lemma
\ref{lemma:min-parabolic} implies $\Gamma$ is minimally parabolic.
Finally, Lemma \ref{lemma:simple-ford-uniform} shows $\HH^3/\Gamma$ is
homeomorphic to the interior of $C$, so this is indeed a
uniformization of $C$.
\end{proof}

We conclude this section by stating a lemma that will help us identify
representations which are \emph{not} discrete.  It is essentially the
Shimizu--Leutbecher lemma \cite[Proposition II.C.5]{maskit}. 

\begin{lemma}
Let $\Gamma$ be a discrete, torsion free subgroup of ${\rm
  PSL}(2,\CC)$ such that $M=\HH^3/\Gamma$ has a rank two cusp.
Suppose that the point at infinity projects to this cusp, and let
$\Gamma_\infty$ be its stabilizer in $\Gamma$.  Then for all $\zeta
\in \Gamma \setminus \Gamma_\infty$, the isometric sphere of $\zeta$
has radius at most the minimal (Euclidean) translation length of all
non-trivial elements in $\Gamma_\infty$.
\label{lemma:not-gf}
\end{lemma}

\section{Algorithm to compute Ford domains}\label{sec:algorithm}

We will use Ford domains to study geometrically finite minimally
parabolic uniformizations of the $(1;2)$--compression body.  To
facilitate this study, we have developed algorithms to construct Ford
domains.  In this section, we present an algorithm which is guaranteed
to construct the Ford domain, but is impractical.  We also present a
practical procedure which we have implemented, which we conjecture
will always construct the Ford domain of the $(1;2)$--compression
body.

\subsection{An initial algorithm}

Let $\Gamma$ be a discrete, geometrically finite subgroup of ${\rm
  PSL}(2,\CC)$ such that $\HH^3/\Gamma$ is homeomorphic to the
interior of the $(1;2)$--compression body.  We will assume that
$\Gamma$ is given by an explicit set of matrix generators.  We now
present an (impractical) algorithm to find the Ford domain of
$\HH^3/\Gamma$.  Assume without loss of generality that in the
universal cover $\HH^3$, the point at infinity is fixed by the rank 2
cusp subgroup, $\Gamma_\infty < \Gamma$.

\begin{algorithm}
Enumerate all elements of the group: $\Gamma =\{g_1,g_2, g_3,
\dots\}$.  Again we assume that each $g_i$ is given as a matrix with
explicit entries.  Step through the list of group elements.  At the
$n$-th step:
\begin{enumerate}
\item Draw isometric spheres corresponding to $g_n$ and $g_n^{-1}$.

\item If these isometric spheres are visible over other previously
  drawn isometric spheres (corresponding to $g_1, \dots, g_{n-1}$ and
  their inverses), check if the object obtained by gluing pairs of
  currently visible, previously drawn isometric spheres via the
  corresponding isometries satisfies the hypotheses of Theorem
  \ref{thm:poincare1}.  

\item If it does satisfy these hypotheses, then by the Poincar{\'e}
  polyhedron theorem, Theorem \ref{thm:poincare2}, the fundamental
  group of the manifold is generated by isometries corresponding to
  face identifications.  Therefore, if we can write the generators of
  $\Gamma$ as words in the isometries of these faces, we will be done,
  by Lemma \ref{lemma:finding-ford}.  Put this manifold into a list of
  manifolds built by repeating the previous two steps.

\item For each manifold in the list of manifolds built by steps (1)
  and (2), we have an enumeration of words in the group elements
  generated by gluing isometries of faces: $L= \{h_1, h_2, \dots\}$.
	\begin{enumerate}
	\item For each generator $g$ of $\Gamma$, step through the first $n$
    words of $L$ to see if $g$ equals one of these words.

	\item If each $g$ can be written as a word in one of the first $n$
    elements of $L$, we are done.  The Ford domain is given by the
    isometric spheres which are the faces of this manifold.
	\end{enumerate}
\end{enumerate}
\label{algorithm:impractical}
\end{algorithm}

Note that in step (2), if we find that isometric spheres glue to give
a manifold, it does not necessarily follow that this manifold is our
original compression body.  For example, we may have found a
non-trivial cover of the original compression body.  Therefore, steps
(3) and (4) are required.

Since Ford domains of geometrically finite hyperbolic manifolds have a
finite number of faces, after a finite number of steps, Algorithm
\ref{algorithm:impractical} will have drawn all isometric spheres
corresponding to visible faces.  Since identifying a finite number of
generators as words in a finite number of generators given by face
pairings can be done in a finite number of steps, after a finite
number of steps the algorithm will terminate.

\subsection{A practical procedure}

The algorithm above is impractical for computer implementation.  In
this section we present a practical procedure, which will generate the
Ford domain and terminate in many cases for a $(1;2)$--compression
body.  We conjecture it will terminate for all cases.

We have implemented this procedure, and used the images it produced to
analyze behavior of paths of Ford domains.  The computer images of
this paper were generated by this program.

\begin{procedure}
Let $\alpha$, $\beta$ be parabolic, fixing a common point at infinity
in $\HH^3$.  Let $\gamma$ be loxodromic, such that $\langle\alpha, \beta,
\gamma\rangle \cong (\ZZ \times \ZZ) \ast \ZZ$.  

Conjugate such that
$$\alpha = \mat{1&a\\0&1}, \quad \beta = \mat{1&b\\0&1}, \quad \gamma
= \mat{c&-1\\1&0}.$$

We will hold two lists: The list of elements to draw, $L_0$, and the
list of elements that have been drawn $L_1$.  These are ordered
lists. 

\emph{Initialization.}  Replace $\alpha$ and $\beta$ if necessary, so
that the lattice generated by $a$ and $b$ has generators of shortest
length. 

Replace $\gamma$ if necessary so that $\gamma(\infty)$ is within the
parallelogram with vertices at $0 = \gamma^{-1}(\infty)$, $a$, $b$,
and $a+b$.

Add $\gamma$ and $\gamma^{-1}$ to the list of elements to draw, $L_0$.

\emph{Loop.}  While the list $L_0$ is non-empty, do the following.
\begin{enumerate}
\item Remove the first element of $L_0$, call it $\zeta$.  Consider
  the isometric sphere of $\zeta$.  Check $I(\zeta)$ against elements
  of $L_1$.  If $I(\zeta)$ is no longer visible, discard and start
  over with the next element of $L_0$.  If $I(\zeta)$ is still
  visible, draw the isometric sphere determined by $\zeta$ to the
  screen.  Add $\zeta$ to the end of the list $L_1$.

  Now also draw isometric spheres of each element of the form
  $w=\alpha^{\epsilon}\beta^{\delta}I(\zeta)$, where $\epsilon$,
  $\delta$ lie in $\{0, \pm 1, \pm 2, \dots \pm m\}$, with $m$ chosen
  so that we draw only those translates of $I(\zeta)$ which are contained
  in the region of the screen.

\item For each $\xi$ in the list of drawn elements $L_1$, find
  integers $p$, $q$ such that the center of $\alpha^p\beta^q I(\zeta)$
  is nearest the center of $\xi$.

  For each isometric sphere of the form
  $\alpha^{p+\epsilon}\beta^{q+\delta}I(\zeta) = I(\zeta
  \beta^{-q-\delta} \alpha^{-p-\epsilon})$, with $\epsilon$, $\delta$
  in $\{0, \pm 1\, \pm 2, \pm 3 \}$, check if that isometric sphere
  and $I(\xi)$ intersect visibly.  That is, check if they intersect
  and, if so, if the edge of their intersection is visible from
  infinity.  (In the case of $I(\zeta)$, no need to check for
  intersections of $I(\zeta)$ and the isometric sphere of the newly
  added last element $\zeta$ of $L_1$.)

We claim that if $I(\xi)$ intersects any translate of $I(\zeta)$ under
$\Gamma_\infty$, then that translate will have the form
$\alpha^{p+\epsilon}\beta^{q+\delta} I(\zeta)$ where $\epsilon$, $\delta$
are in $\{0, \pm 1, \pm 2, \pm 3 \}$.  See Lemma
\ref{lemma:translate-zeta} below.

\item If $\alpha^{p+\epsilon} \beta^{q+\delta} I(\zeta)$ and $I(\xi)$
  do intersect visibly, then the isometric sphere of the element $\xi
  w^{-1}$ should be drawn, where $w = \zeta \beta^{-q-\delta}
  \alpha^{-p-\epsilon}$, so that $I(w) = \alpha^{p+\epsilon}
  \beta^{q+\delta} I(\zeta)$.  Step through the lists $L_1$ and $L_0$
  to ensure the isometric sphere $I(\xi w^{-1})$ hasn't been drawn
  already, and is not yet slated to be drawn (to avoid adding the same
  sequence of faces repeatedly -- note there are more time effective
  ways of ensuring the same thing).  If $\xi w^{-1}$ is not in either
  list, then add $\xi w^{-1}$, and $w \xi^{-1}$ to the end of the list
  $L_0$ to be drawn.
\end{enumerate}
\label{procedure}
\end{procedure}

\begin{lemma}\label{lemma:translate-zeta}
  Suppose $\alpha$ and $\beta$ are parabolic fixing the point at
  infinity, chosen as above such that $\alpha$ has the shortest
  translation length in the group $\langle \alpha,\beta\rangle \cong
  \ZZ \times \ZZ$, and such that $\beta$ has the shortest translation
  length of all parabolics independent from $\alpha$.  Suppose $\xi$
  and $\zeta$ are loxodromic such that the group $\langle \alpha,
  \beta, \xi, \zeta \rangle$ is discrete.  Choose integers $p,q$ such
  that the center of $I(\xi)$ is nearer the center of $ \alpha^p
  \beta^q I(\zeta)$ than the center of any other translate of
  $I(\zeta)$ under $\langle \alpha, \beta\rangle$.  Then if $I(\xi)$
  intersects any translate of $I(\zeta)$, that translate must be of
  the form $\alpha^{p + \epsilon} \beta^{q + \delta} I(\zeta)$ for
  $\epsilon, \delta \in \{0, \pm 1, \pm 2, \pm 3\}$.
\end{lemma}

\begin{proof}
Apply an isometry to $\HH^3$ so that $\alpha$ translates by exactly
$1$ along the real axis in $\CC$.  Note that after this isometry, by
Lemma \ref{lemma:not-gf}, all isometric spheres have radius at most
$1$.  Hence if two intersect, the distance between their centers is
less than $2$.  Let $x$ denote the center of $\alpha^p \beta^q
I(\zeta)$.  We may apply another isometry of $\HH^3$ so that $x=0$ in
$\CC$.  Finally, since $\beta$ is the shortest translation independent
of $\alpha$, $\beta$ must translate $x$ to be within the hyperbolic
triangle on $\CC$ with vertices $1/2 + i\sqrt{3}/2$, $-1/2 +
i\sqrt{3}/2$, $\infty$.  

Since the center of $I(\xi)$, denote it by $y$, is closer to $x$ than
to any of the translates of $x$ under $\langle \alpha, \beta \rangle$,
the real coordinate of $y$ in $\CC$ must have absolute value at most
$1/2$.  Similarly, the difference in imaginary coordinates of $y$ and
$\beta x$ is at least $\sqrt{3}/6$, for otherwise the square of the
distance between $y$ and some lattice point of the form
$\alpha^\epsilon \beta x$ is at most $(1/2)^2 + (\sqrt{3}/6)^2 = 1/3$.
Finally, we may assume the imaginary coordinate of $y$ is positive, by
symmetry of the lattice.

Suppose $I(\xi)$ meets $\alpha^{p+\epsilon}\beta^{q+\delta}I(\zeta)$,
where one of $|\epsilon|$ or $|\delta|$ is greater than $3$.  Then the
distance between $y$ and $\alpha^\epsilon \beta^\delta x$ on $\CC$ is
at most $2$.  On the other hand, if $|\delta| \geq 3$, then the
difference between the imaginary coordinates of $y$ and
$\alpha^\epsilon \beta^\delta x$ is at least $\sqrt{3} + \sqrt{3}/6 >
2$, which is a contradiction.  So suppose $|\delta| < 3$ and
$|\epsilon|>3$.  Then the difference in real coordinates of
$\alpha^\epsilon \beta^\delta x$ and $y$ is at least $4 - 1/2 -
\delta\cdot 1/2 > 2$, which is again a contradiction.
\end{proof}

\begin{theorem}
Suppose each of the spheres drawn by Procedure \ref{procedure} is a
face of the Ford domain of a geometrically finite uniformization of
the $(1;2)$--compression body $C$.  Then the procedure draws (at least
one translate under $\Gamma_\infty$ of) all visible isometric spheres,
and the procedure terminates.
\label{thm:algorithm}
\end{theorem}

\begin{proof}
The fact that the procedure terminates follows from Corollary
\ref{cor:bowditch}: there are only finitely many visible faces, and
each face the procedure draws is visible.

The fact that the procedure draws all visible isometric spheres of the
Ford domain will follow from Lemma \ref{lemma:finding-ford} and the
Poincar{\'e} polyhedron theorem, as follows.

First, suppose the faces corresponding to $\gamma$ and $\gamma^{-1}$
are visible, and they do not intersect each other or any other faces.
Then the procedure terminates after drawing these faces and a few
translates under $\Gamma_\infty$.  Because there are no edges of
intersection, the argument of Lemma \ref{lemma:simple-ford} implies
that the only visible face of the Ford domain corresponds to $\gamma$
(and $\gamma^{-1}$), and in this case we are done.

So suppose two isometric spheres drawn by the procedure intersect.
Say isometric spheres $I(g)$ and $I(h)$ intersect.  Then the procedure
will draw $I(gh^{-1})$.  Since the procedure only draws visible
isometric spheres, $I(gh^{-1})$ must be visible.  By Lemma
\ref{lemma:edge-visible}, it intersects $I(h^{-1})$ in an edge which
is mapped isometrically to the edge of $I(g) \cap I(h)$.  Changing
roles of $g$ and $h$ in the same lemma, the isometric sphere
$I(hg^{-1})$ must be visible, and $I(hg^{-1}) \cap I(g^{-1})$ is
mapped isometrically to $I(g) \cap I(h)$.

Now notice that the faces of the Ford domain corresponding to the
pairs $I(g)$ and $I(g^{-1})$, $I(h)$ and $I(h^{-1})$, and $I(gh^{-1})$
and $I(hg^{-1})$ are the only faces that meet the edge class of
$I(g)\cap I(h)$ (up to translation by $\Gamma_\infty$).  This can be
seen by noting that $g$ takes $I(g) \cap I(h)$ and $I(h)$ to
$I(g^{-1}) \cap I(hg^{-1})$ and $I(hg^{-1})$, respectively.  Then
apply $hg^{-1}$.  This sends $I(g^{-1}) \cap I(hg^{-1})$ and
$I(g^{-1})$ to $I(h^{-1}) \cap I(gh^{-1})$ and $I(h^{-1})$,
respectively.  Finally apply $h^{-1}$, which sends $I(h^{-1}) \cap
I(gh^{-1})$ and $I(gh^{-1})$ to $I(h) \cap I(g)$ and $I(g)$,
respectively.  Thus the monodromy is given by $h^{-1} \circ hg^{-1}
\circ g = 1$.  As for dihedral angles around this edge class, because
the monodromy is the identity, the sum of the dihedral angles must be
a multiple of $2\pi$.  Since there are only three faces in the edge
class, and the dihedral angle between any two faces is less than
$\pi$, the sum of the dihedral angles around the edge $I(h) \cap I(g)$
must be exactly $2\pi$.  Now we have the hypotheses of the
Poincar{\'e} polyhedron theorem, Theorem \ref{thm:poincare1}.  That
theorem tells us that the gluing of the faces our procedure has drawn
gives a smooth hyperbolic manifold.  Lemma \ref{lemma:finding-ford}
implies that the procedure has drawn the entire Ford domain, as
desired.
\end{proof}

One way the hypotheses of Theorem \ref{thm:algorithm} might not hold
is if there is an edge class of the cell structure on the Ford domain
that meets more than three visible faces.  When two of the visible
faces intersect, say corresponding to $I(g)$ and $I(h)$, our procedure
will draw $I(gh^{-1})$.  However, if the edge class meets more than
three visible faces, the isometric sphere $I(gh^{-1})$ will not be
visible, and so the hypotheses of the theorem are not satisfied.  In
practice, we were unable to find a structure on the
$(1;2)$--compression body for which this situation arose.  S.~Burton
found such a structure on a $(1;3)$--compression body
\cite{burton:thesis}.  However, even in this higher genus case the
above procedure drew all visible isometric spheres for the example,
since the isometric sphere covering $I(gh^{-1})$ arose as the
intersection of other visible isometric spheres.  Based on
experimental evidence in the case of the $(1;2)$--compression body, we
offer the following conjecture.

\begin{conjecture}\label{conj:procedure}
  Procedure \ref{procedure} always draws the Ford domain for a
  geometrically finite uniformization for the $(1;2)$--compression
  body, and terminates.
\end{conjecture}

The generalization of Procedure \ref{procedure} to
$(1;n)$--compression bodies, for $n\geq 3$ has been shown to fail
by S.~Burton \cite{burton:thesis}.  That is, the procedure will not
necessarily draw the full Ford domain.  This is because in the higher
genus case, a choice of loxodromic generators may give an isometric
sphere which is completely covered by some visible isometric sphere.
As long as that visible isometric sphere is not one of our generators,
and as long as the isometric spheres of our generators remain disjoint
from that visible isometric sphere, the visible isometric sphere will
never be drawn by the above procedure.  However in the
$(1;2)$--compression body case, up to translation by $\Gamma_\infty$
there is only one choice for loxodromic generator, and so this issue
does not seem to arise.

\section{Examples of Ford domains}\label{sec:ford}
Recall that we are interested in isotopy classes of the core tunnel of
a $(1;2)$--compression body.  We use the computer program implementing
Procedure \ref{procedure} to study isotopy classes of the core tunnel
for many different geometrically finite uniformizations.  To identify
core tunnels in Ford domains, we will examine the dual structure to a
Ford domain.  In this section, we define the dual structure and
present several examples.  The examples were obtained by computer
using the procedure of the previous section.


\subsection{Paths of Ford domains}

Recall that if $C$ denotes the $(1;2)$--compression body, then
$\pi_1(C) \cong (\ZZ\times \ZZ)\ast \ZZ$ with generators we denote
$\alpha$ and $\beta$ for the $(\ZZ\times \ZZ)$ factor, and $\gamma$.
Let $\rho_0\co \pi_1(C) \to PSL(2,\CC)$ be the representation of
Example \ref{ex:simple-ford}.  Keeping the images of $\alpha$ and
$\beta$ parabolic, allow the images of the three generators $\alpha$,
$\beta$, and $\gamma$ to vary smoothly.  We obtain a smooth path of
discrete, faithful representations $\rho_t$.  For some amount of time,
these will be minimally parabolic geometrically finite uniformizations
of $C$.  As $\rho_t$ changes smoothly, the visible isometric spheres
of $\HH^3/\rho_t(\pi_1(C))$ will change smoothly.  In particular, we
can change the images of the generators such that two isometric
spheres bump into each other.  By Lemma \ref{lemma:edge-visible}, if
two visible isometric spheres intersect, then a new visible face must
arise when they meet.  We present two examples to illustrate some of
the behavior that may occur.

\begin{example}
Consider the smooth path of representations $\rho_t\co \pi_1(C) \to
PSL(2,\CC)$ given by
$$\rho_t(\alpha) = \mat{1&5+i\\ 0&1}, \quad \rho_t(\beta) = \mat{1&
  5.5i\\ 0&1}, \quad \rho_t(\gamma) = \mat{-1+i\,t& -1\\ 1& 0},$$
where $t$ runs from $2$ down to $1.2$.

Note here that $\rho_t(\alpha)$ and $\rho_t(\beta)$ are constant.
They were chosen somewhat arbitrarily to be parabolics fixing
infinity, with large enough Euclidean translation distance that
nontrivial translations under $\Gamma_\infty = \langle \rho_t(\alpha),
\rho_t(\beta) \rangle$ of the isometric spheres corresponding to
$\rho_t(\gamma^{\pm 1})$ and $\rho_t(\gamma^{\pm 2})$ don't meet any
of these original isometric spheres.

Consider the isometric spheres corresponding to $\rho_t(\gamma)$.  By
Lemma \ref{lemma:iso-center-rad}, these have radius $1$ throughout the
path.  When $t=2$, the isometric spheres of $\rho_t(\gamma)$ and
$\rho_t(\gamma^{-1})$, which have centers $0$ and $-1+i\,t$
respectively, do not intersect, so we have the simple Ford spine with
a single face as above.  However, as $t$ decreases, these two
isometric spheres first become tangent, at $t=\sqrt{3}$, and then
overlap for $t<\sqrt{3}$.  As these spheres meet, the isometric
spheres corresponding to $\rho_t(\gamma^2)$ and $\rho_t(\gamma^{-2})$
emerge, and their intersections with isometric spheres of
$\rho_t(\gamma)$ and $\rho_t(\gamma^{-1})$, respectively, become
visible, as predicted by Lemma \ref{lemma:edge-visible}.  We can
compute explicitly that for these particular representations, for $1.2
< t < \sqrt{3}$, the region cut out by the isometric spheres of
$\rho_t(\gamma^{\pm 1})$ and $\rho_t(\gamma^{\pm 2})$ and a vertical
fundamental domain for $\Gamma_\infty$ is a fundamental polyhedron for
a manifold, using Poincar{\'e}'s theorem \ref{thm:poincare1}.  By
Lemma \ref{lemma:finding-ford}, these isometric spheres must define
the Ford domain for the manifold $\HH^3/\rho_t(\pi_1(C))$.  Thus our
Ford spine has two faces, corresponding to $\rho_t(\gamma)$ and
$\rho_t(\gamma^2)$.  Figure \ref{fig:bumping-1} illustrates this
particular example.

We claim this is still a uniformization of $C$, i.e.\ that
$\HH^3/\rho(\pi_1(C))$ is homeomorphic to the interior of $C$.  The
Ford spine of $\HH^3/\rho(\pi_1(C))$ has two faces, one of which has
boundary which is the union of the 1--cell of the spine and an arc on
$\partial_+C$ (corresponding to $\gamma^{\pm 2}$).  Collapse the
1--cell and this face.  The result is a new complex with the same
regular neighborhood.  It now has a single 2--cell attached to $\D_+
C$.  Thus $\HH^3/\rho(\pi_1(C))$ is obtained by attaching a 2--handle
to $\D_+C \times I$, and then removing the boundary.  So
$\HH^3/\rho(\pi_1(C))$ is homeomorphic to the interior of $C$.
\label{ex:bumping}
\end{example}

\begin{figure}
\begin{center}
\begin{tabular}{cccc}
	\makebox{\begin{tabular}{c}
		\includegraphics{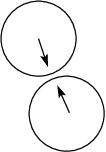} 
		\end{tabular}}
	\makebox{\begin{tabular}{c}
			\includegraphics{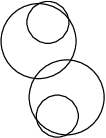}
	   \end{tabular}}
	& \hspace{1in} &
	\makebox{\begin{tabular}{c}
			\vspace{-.1in} \\
			\includegraphics[width=1.75in]{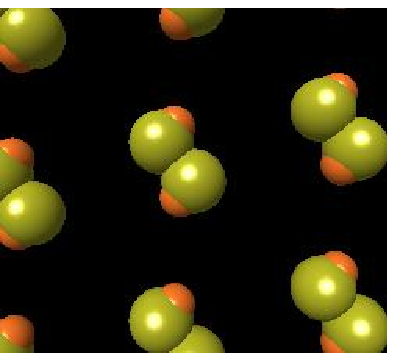}
			\end{tabular}}
\end{tabular}
\end{center}
\caption{Faces of the Ford domain meet.  Left: schematic picture for
	$t=2$ down to $t=1.2$.  Right: Computer generated image for
	$t=1.2$.}
\label{fig:bumping-1}
\end{figure}

\begin{example}
Consider the same path as in Example \ref{ex:bumping}, only now allow
$t$ to run from $1.2$ down to $0.8$.  As $t$ decreases, the isometric
spheres corresponding to $\rho_t(\gamma^{\pm 2})$ slide towards those
corresponding to $\rho_t(\gamma^{\pm 1})$, as illustrated in Figure
\ref{fig:sliding1}.  At approximately time $t=1$, these isometric
spheres meet visibly, and for $1 > t > 0.8$, these isometric spheres
overlap.  The isometric spheres corresponding to $\rho_t(\gamma^{\pm
  3})$ are visible during these times, and emerge out from under the
intersection between faces corresponding to $\rho_t(\gamma^{\pm 1})$
and $\rho_t(\gamma^{\pm 2})$, as illustrated in Figure
\ref{fig:sliding1}.  Again one may show that these isometric spheres,
as well as a vertical fundamental domain for $\Gamma_\infty$, cut out
a polyhedron which glues up to give our manifold
$\HH^3/\rho_t(\pi_1(C))$, so again by Lemma \ref{lemma:finding-ford},
these isometric spheres cut out the Ford domain for the manifold.

We can show that this is a uniformization of $C$, i.e.\ that
$\HH^3/\rho(\pi_1(C))$ is homeomorphic to the interior of $C$, this
time by considering the face of the Ford spine corresponding to
$\gamma^{\pm 3}$.  This face has boundary consisting of two 1--cells
and an arc on $\D_+C$.  Collapse this face.  In fact, we may collapse
the faces in the order they appeared, and we are again left with a
single 2--cell attached to $\D_+C$ (corresponding to $\gamma^{\pm
  1}$).  So again this is a uniformization of $C$.
\label{ex:sliding}
\end{example}

\begin{figure}
\begin{center}
\begin{tabular}{cccc}
	\makebox{\begin{tabular}{c}
			\includegraphics{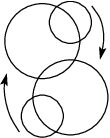}
	\end{tabular}} &
	\makebox{\begin{tabular}{c}
			\includegraphics{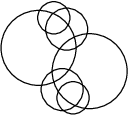}
	\end{tabular}} &
	\hspace{1in} &
	\makebox{\begin{tabular}{c}
			\vspace{-.1in} \\
      \includegraphics[width=1.75in]{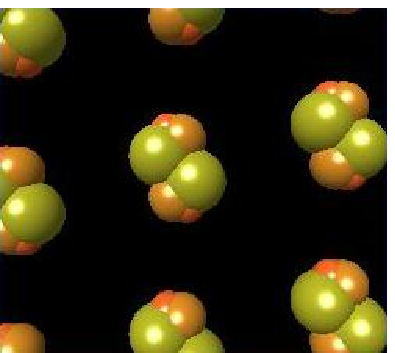}
		\end{tabular}}
\end{tabular}
\end{center}
\caption{Left:  Schematic picture of path for $t=1.2$ down to
	$t=0.8$.  Right: Computer generated image, $t=0.8$.}
\label{fig:sliding1}
\end{figure}

The examples above illustrate the phenomenon of Lemma
\ref{lemma:edge-visible}, that is, that new faces emerge when existing
faces meet in a path of uniformizations.  We will see in
Section \ref{sec:geodesics} that this is the only way a new face can
emerge.

\subsection{The dual structure}

Recall that we are interested in core tunnels of the
$(1;2)$--compression body $C$.  In many cases, we can identify the
core tunnel as an edge of the geometric dual of the Ford spine.  This
dual is reminiscent of the canonical polyhedral decompositions for
finite volume manifolds which were introduced by Epstein and Penner
\cite{epstein-penner}.  We build the dual structure as follows.

Consider again $\F = \HH^3 \setminus \bigcup_{g \in \Gamma\setminus
  \Gamma_\infty} B(g)$.  To each visible isometric sphere $I(g)$ of
$\F$, there is an associated edge $e(g)$, which is the geometric dual
of $I(g)$ running from the center of $I(g)$ to infinity in $\HH^3$.

If two isometric spheres $I({g_1})$ and $I({g_2})$ of $\F$ overlap
visibly, then they correspond to a dual face $F({g_1,g_2})$ which is
the vertical plane bounded by $e(g_1)$ and $e({g_2})$ intersected with
$\F$.

If visible isometric spheres of $\F$ meet (visibly) in a vertex, then
their dual is a 3--dimensional region in $\F$ bounded by dual faces.

This forms a complex $C$.  When we take $C/\Gamma$, we obtain a
complex $C_0$ which is the geometric dual of the Ford spine.

\begin{example}
Consider Example \ref{ex:simple-ford}, which gives a minimally
parabolic geometrically finite uniformization on a
$(1;2)$--compression body with only one face of the Ford spine.  The
geometric dual to the Ford spine for this example is a single edge
running through the geometric center of the Ford spine.  This edge
lifts to a collection of geodesics in $\F \subset \HH^3$ running
through centers of isometric spheres corresponding to $\rho(\gamma)$,
$\rho(\gamma^{-1})$, and their translates under $\Gamma_\infty$.  See
Figure \ref{fig:simple-dual}.
\label{ex:simple-spine}
\end{example}

\begin{figure}
\begin{center}
	\input{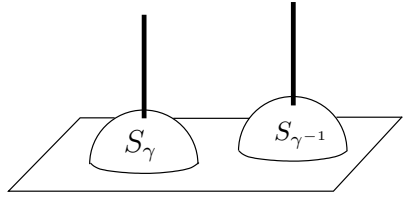}
\end{center}
\caption{The dual to the simplest Ford spine is an edge that lifts to
	a collection of vertical geodesics in $\F$, shown in bold.}
\label{fig:simple-dual}
\end{figure}

\begin{example}
Consider again Example \ref{ex:bumping}, which describes the Ford
domain of a geometrically finite unifomization of $C$ in which the
isometric spheres corresponding to $\rho(\gamma)$ and
$\rho(\gamma^{-1})$ ``bump'', and only isometric spheres corresponding
to $\rho(\gamma^2)$ and $\rho(\gamma^{-2})$ emerge.  Consider the
geometric dual to this picture.  In $\F/\Gamma_\infty$, we see three
intersections of isometric spheres: one corresponding to
$\rho(\gamma)$ and $\rho(\gamma^{-1})$, one corresponding to
$\rho(\gamma^2)$ and $\rho(\gamma)$, and one corresponding to
$\rho(\gamma^{-1})$ and $\rho(\gamma^{-2})$.  Thus the lift of the
geometric dual to $\F$ has the form on the left of Figure
\ref{fig:bump-dual}.

These three lines of intersection in $\F$ are all glued under the
action of $\Gamma$ to the same single line.  The dual faces glue up to
give a single ideal triangle, as on the right in Figure
\ref{fig:bump-dual}, with two sides on the same edge (dual to the
isometric sphere of $\gamma$).
\label{ex:bump-dual}
\end{example}

\begin{figure}
\begin{center}
	\includegraphics{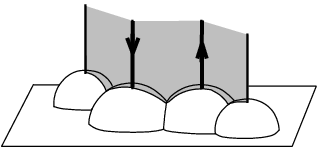}
	\hspace{.5in}
	\includegraphics{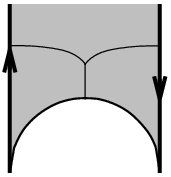}
\end{center}
\caption{Left: the lift to $\F$ of the geometric dual of a Ford spine
	as in Example \ref{ex:bumping}.  Right: in this case the geometric
	dual to the Ford spine is a single ideal triangle, with two sides on
	the same edge.}
\label{fig:bump-dual}
\end{figure}

\begin{example}
When $\Gamma$ is the final uniformization in the path of
representations considered in Example \ref{ex:sliding}, the dual is a
single ideal tetrahedron, as shown in Figure \ref{fig:slide-dual}.
Note the tetrahedron has two faces which are identified to each other
under the action of $\Gamma$.
\label{ex:slide-dual}
\end{example}

\begin{figure}
\begin{center}
	\includegraphics{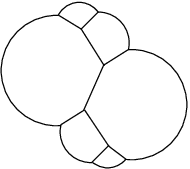}
	\hspace{.5in}
	\includegraphics{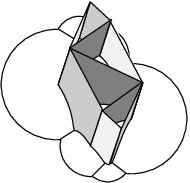}
\end{center}
\caption{On the right is the lift to $\F$ of the geometric dual of the
  Ford spine of Example \ref{ex:sliding}.  The dual structure meets
  the horosphere about infinity in four triangles, corresponding to
  the four vertices of the single ideal tetrahedron.}
\label{fig:slide-dual}
\end{figure}

The dual structure, along with a horoball at infinity, also carries
the topological information of the $(1;2)$--compression body.

\begin{lemma}
For $M$ the interior of any hyperbolizable 3--manifold with a single
torus boundary component, let $\rho\co \pi_1(M) \to PSL(2,\CC)$ be a
minimally parabolic geometrically finite uniformization of $M$.  Then
there is a deformation retraction of $M$ onto the union of the
geometric dual of its Ford spine and an embedded horoball neighborhood
of the rank 2 cusp.
\label{lemma:retract-dual}
\end{lemma}

\begin{proof}
Because $\rho$ is geometrically finite, there exist finitely many
visible isometric spheres in a Ford domain, which we view as $\F
\subset \HH^3$ intersected with a vertical fundamental domain.  The
boundaries of these isometric spheres are circles on $\CC$, which
bound disks on $\CC$.  There exists some $\epsilon>0$ such that the
$\epsilon$--neighborhood of the union of these disks on $\CC$ is
embedded in $\CC$.  Translates by $\Gamma_\infty$ remain embedded on
$\CC$.  Now let $H_\infty$ be the lift an embedded horoball
neighborhood of the rank 2 cusp to $\HH^3$.  Project the
$\epsilon$--neighborhood of the union of disks vertically onto
$\partial H_\infty$.  For each visible isometric sphere, there is a
portion of a Euclidean cone in $\HH^3 \setminus H_\infty$ which
intersects $\CC$ in the boundary of the isometric sphere, and
intersects $\partial H_\infty$ in the $\epsilon$--neighborhood.  Let
$S$ denote the union of all these cones.  Note they form a regular
neighborhood of the lift of the geometric dual of the Ford spine,
intersected with $\F \setminus H_\infty$.

For the first step of the deformation retract, consider a point $x$ in
$\HH^3 \setminus (S \cup H_\infty)$.  Hyperbolic space $\HH^3$ is
foliated by vertical lines, and the vertical line through $x$ will
meet $\partial (H_\infty \cup S)$ in exactly one point.  We define a
deformation retract on $\HH^3 \setminus (S\cup H_\infty)$ by taking
$x$ to this unique point on $\partial (H_\infty \cup S)$.  

For the second step, since $S$ is a regular neighborhood of the lift
of the geometric dual of the Ford spine in $\F \setminus H_\infty$, we
deformation retract $S \cup \partial H_\infty$ to the union of the
geometric dual and the boundary $\partial H_\infty$.  We may choose
the deformation retraction to be equivariant with respect to the
action of $\rho(\pi_1(C))$.  Putting both steps together and taking
the quotient under $\rho(\pi_1(C))$, the result is the desired
deformation retraction of $\HH^3/\rho(\pi_1(C))$.
\end{proof}

With this picture of the dual structure, the fact that the core tunnel
is geodesic in the case in which the Ford spine consists of a single
face is immediate.

\begin{prop}
Suppose the Ford spine of a minimally parabolic geometrically finite
hyperbolic uniformization of a $(1;2)$--compression body consists of a
single face, corresponding to the loxodromic generator.  Then the core
tunnel is isotopic to a geodesic, dual to this single face.
\label{prop:simple-ford}
\end{prop}

\begin{proof}
Let $\rho\co\pi_1(C) \to PSL(2,\CC)$ be a uniformization of $C$ with
one face of the Ford spine, as in the statement of the proposition,
and denote $\rho(\pi_1(C))$ by $\Gamma$.  As in Example
\ref{ex:simple-spine}, the dual structure to the Ford spine consists
of a single edge.

By Lemma \ref{lemma:retract-dual}, we may retract $\HH^3/\Gamma$ onto
a union of a horoball neighborhood of the cusp and this geodesic.
Thus in this case, the single geodesic, which is the edge dual to the
single face of the Ford spine, is isotopic to the core tunnel.
\end{proof}

In fact, for any uniformization $\rho\co \pi_1(C) \to PSL(2,\CC)$, the
core tunnel will always be \emph{homotopic} to the edge dual to the
isometric sphere corresponding to $\rho(\gamma)$.

\begin{lemma}
For any uniformization $\rho\co \pi_1(C) \to PSL(2,\CC)$, the core
tunnel will be homotopic to the edge dual to the isometric sphere
corresponding to the loxodromic generator of $\rho(\pi_1(C))$.
\label{lemma:homotopic-core}
\end{lemma}

\begin{proof}
Denote the loxodromic generator by $\rho(\gamma)$.  Consider the core
tunnel in the compression body $\HH^3/\rho(\pi_1(C))$.  Take a
horoball neighborhood $H_\infty$ of the cusp.  The core tunnel runs
through the horospherical torus $\D H_\infty$ into the cusp.  Denote
by $\widetilde{H}_\infty$ a lift of $H_\infty$ to $\HH^3$ about the
point at infinity in $\HH^3$.

There is a homeomorphism from $C\setminus \bdy_+ C$ to
$(\HH^3/\rho(\pi_1(C))) \setminus \mathring{H}_\infty$.  Slide the
tunnel in $C$ so that it starts and ends at the same point, and so
that the resulting loop represents $\gamma$.  The image of this loop
under the homeomorphism to $(\HH^3/\rho(\pi_1(C))) \setminus
\mathring{H}_\infty$ is some loop.  It lifts to an arc in $\HH^3$
starting on $\widetilde{H}_\infty$ and ending on
$\rho(\gamma)(\widetilde{H}_\infty)$.  Extend to an arc in $\HH^3 /
\rho(\pi_1(C))$ by attaching a geodesic in $\widetilde{H}_\infty$ and
in $\rho(\gamma)(\widetilde{H}_\infty)$ and projecting.  This is
isotopic to (the interior of) the core tunnel.  Now homotope the arc
to a geodesic.  It will run through the isometric sphere corresponding
to $\rho(\gamma^{-1})$ once.
\end{proof}

\section{Paths of structures and tunnels}\label{sec:geodesics}
We have encountered examples of minimally parabolic geometrically
finite uniformizations of a $(1;2)$--compression body $C$ for which
the core tunnel is geodesic.  This was shown explicitly for structures
with simple Ford spines in Proposition \ref{prop:simple-ford}.  It can
also be seen for those with spines as in Examples \ref{ex:bumping} and
\ref{ex:sliding}, by constructing a deformation retract onto the
geodesic dual to the face corresponding to $\gamma$.

In this section we investigate Conjecture \ref{conj:core-tunnel} more
carefully.  We find families of geometrically finite uniformizations
of $C$ for which the core tunnel is geodesic.  Those structures of
Examples \ref{ex:bumping} and \ref{ex:sliding} will fit into these
families.

Our method of proof is to consider paths through the space of
minimally parabolic geometrically finite uniformizations, and the
corresponding Ford spines and their dual structures.  We will see that
in many cases, under some assumptions on the path, the core tunnel
must remain isotopic to a geodesic.

\subsection{Paths and visible isometric spheres}

In this subsection we will work with slightly more general manifolds
than $C$.  We let $M$ be the interior of a hyperbolic manifold with
only one of its boundary components a torus.  

The following follows from work of Bers, Kra, and Maskit (see
\cite{bers74}).

\begin{lemma}
The space of minimally parabolic geometrically finite uniformizations
of $M$ is path connected.
\label{lemma:path-conn}
\end{lemma}

\begin{proof}
Bers, Kra, and Maskit showed that the space of conjugacy classes of
minimally parabolic geometrically finite uniformizations may be
identified with the Teichm\"uller space of the higher genus boundary
components, quotiented out by ${\rm Mod}_0(M)$, the group of isotopy
classes of homeomorphisms of $M$ which are homotopic to the identity.
Since the Teichm\"uller space is path connected, the quotient will
also be path connected.
\end{proof}

Thus given any minimally parabolic geometrically finite uniformization
of $C$, it is connected by a path of uniformizations to a
uniformization admitting a simple Ford spine, as in
Lemma \ref{lemma:simple-ford}.

Now, we will be taking paths through the interior of the space of
geometrically finite, minimally parabolic uniformizations of the
manifold $M$.  Technically, such uniformizations are paths of
representations $\rho_t\co \pi_1(M) \to PSL(2,\CC)$.  For any group
element $g \in \pi_1(M)$, $\rho_t$ will give a path of isometric
spheres corresponding to $\rho_t(g)$.

As the isometric spheres in a Ford domain bump into each other, new
isometric spheres become visible, and in turn visible faces may become
invisible.  We will determine when and how spheres become visible.
First, we show that isometric spheres are visible for an open set of
time.

\begin{lemma}
Let $\Gamma$ be a group with subgroup $\Gamma_\infty \cong \ZZ\times
\ZZ$, and let $\rho_t\co \Gamma \to \PSL(2,\CC)$ be a continuous path of
minimally parabolic geometrically finite representations of $\Gamma$
such that $\rho_t(\Gamma_\infty)$ fixes the point at infinity in
$\HH^3$ for all $t$.  Then any isometric sphere will be visible for an
open set of time.
\label{lemma:open}
\end{lemma}

\begin{proof}
Suppose the isometric sphere corresponding to the element $g_0 \in
\Gamma$ is visible at time $t_0$.  By Lemma
\ref{lemma:visible-strict-inequality}, there exists $x$ on the
hemisphere $I({\rho_{t_0}(g_0)})$ which is not contained in the closure
of half--spaces $B({\rho_{t_0}(h)})$ bounded by any isometric spheres
corresponding to elements of $\Gamma \setminus (\Gamma_\infty \cup
\Gamma_\infty g_0)$.  Let $U$ be a small open ball around $x$ which is
disjoint from the closures of these half spaces.  

We claim that there is some $\epsilon > 0$ such that for any $t \in
(t_0 - \epsilon, t_0 + \epsilon)$, $B({{\rho}_{t}(h)}) \cap U =
\emptyset$ for all $h \in \Gamma \setminus (\Gamma_\infty \cup
\Gamma_\infty g_0)$.  We may also choose $\epsilon > 0$ so that
$I({\rho_t(h)}) \cap U \not= \emptyset$ for all $t \in (t_0 - \epsilon,
t_0 + \epsilon)$.  Hence this claim will prove the lemma, because then
points in this intersection will be visible.

Suppose that the claim is not true. There is then a sequence of times
$t_n$ (where $n \geq 1$) tending to $t_0$ and a sequence of elements
$g_n \in \Gamma \backslash (\Gamma_\infty \cup \Gamma_\infty g_0)$
such that $B({{\rho}_{t_n}(g_n)}) \cap U \not= \emptyset$. So
$\rho_{t_n}(g_n)$ lies in the subset $V$ of ${\rm PSL}(2, {\mathbb C})$
defined as follows:
$$
V = \{ g \in {\rm PSL}(2, {\mathbb C})
: \overline{B(g)} \cap \overline{U} \not = \emptyset \hbox{ and }
g^{-1}(H) \cap H = \emptyset \},
$$
where as usual, $H$ denotes an embedded horoball about infinity.

We wish to argue by compactness.  Note that $V$ itself is not compact,
for if $g \in V$, then so is $wg$ for any $w \in \Gamma_\infty$.
However, we may consider a compact subset of $V$.  Let $V_{\rm norm}$
consist of $wg \in \PSL(2, \CC)$ where $g \in V$ and
$w \in \Gamma_\infty$ is chosen such that $I(wg)$ and $I((wg)^{-1})$
have minimal (Euclidean) distance.  That is, for any other $x \in
\Gamma_\infty$, the distance between $I(xg)$ and $I((xg)^{-1})$ is at
least as large as that between $I(wg)$ and $I(wg)^{-1})$. 

Now $V_{\rm norm}$ is a compact subset of ${\rm PSL}(2, {\mathbb C})$.
By composing with a suitable element of $\Gamma_\infty$, we may assume
that each $\rho_{t_n}(g_n)$ lies in $V_{\rm norm}$.  Hence we may
pass to a subsequence where $\rho_{t_n}(g_n)$ converges to some $h \in
{\rm PSL}(2, {\mathbb C})$.  Now the groups $\rho_{t_n}(\Gamma)$
converge algebraically to $\rho_{t_0}(\Gamma)$.  Since
$\rho_{t_0}(\Gamma)$ is geometrically finite, this convergence is also
geometric \cite{brock-souto}.

So $h$ lies in $\rho_{t_0}(\Gamma)$. Say that $h = \rho_{t_0}(g)$ for
some $g \in \Gamma$. Then $\rho_{t_n}(g g_n^{-1})$ is an element of
$\rho_{t_n}(\Gamma)$ that can be made arbitrarily close to the
identity in ${\rm PSL}(2, {\mathbb C})$ by taking large $n$. Powers of
this form a cyclic subgroup of ${\rm PSL}(2, {\mathbb C})$, and after
passing to a subsequence, these converge geometrically to a
non-discrete subgroup of ${\rm PSL}(2, {\mathbb C})$.  But this implies
that $\rho_{t_0}(\Gamma)$ is not discrete, which is a contradiction.
This proves the claim and hence the lemma.
\end{proof}

In what follows, we will analyze how the pattern of visible isometric
spheres changes along a path $\rho_t(\Gamma)$ of minimally parabolic
geometrically finite uniformizations.  The first step is to examine
how two Euclidean hemispheres $I({\rho_{t}(g_1)})$ and
$I({\rho_t(g_2)})$ interact.  It would be useful to know that during
an interval $[t_-, t_+]$ of time, the set of times where
$I({\rho_{t}(g_1)})$ completely covers $I({\rho_t(g_2)})$ is a finite
collection of closed intervals.  However, this need not be the case in
general.  Although the set of times where $I({\rho_t(g_1)})$ covers
$I({\rho_t(g_2)})$ is a closed subset of $[t_-, t_+]$, this subset can
have infinitely many components.  To visualise this, imagine a
continuous function $[t_-, t_+] \rightarrow {\RR}$ which fluctuates
between positive and negative values infinitely often near some
$t_0 \in [t_-,t_+]$. We may find a path of uniformizations where the
distance of $I({\rho_{t}(g_1)})$ below (or above) $I({\rho_t(g_2)})$
is equal to this function.  Even if we require our path $\rho_t$ of
representations to be smooth, this phenomenon can occur.  However, it
is does not arise when the path of representations $[t_-, t_+] \times
\Gamma \rightarrow {\rm PSL}(2, {\mathbb C})$ is real analytic. Note that
${\rm PSL}(2, {\mathbb C})$ inherits an obvious real analytic structure
from ${\mathbb C}^4$.  Moreover, any path of minimally parabolic
geometrically finite uniformizations can be approximated by a real
analytic path, by the Whitney Approximation Theorem.

\begin{lemma}\label{lemma:real-analytic}
Let $\Gamma$ be a group with a subgroup $\Gamma_\infty \cong {\mathbb Z}
\times {\mathbb Z}$.  Let $\rho_t$ be a real analytic path of
uniformizations of $\Gamma$, where $t \in [t_-, t_+]$, such that
$\rho_t(\Gamma_\infty)$ fixes the point at infinity in $\HH^3$ for all
$t$.  Let $g_1$ and $g_2$ be elements of $\Gamma \setminus
\Gamma_\infty$.  Then, the set of times $t$ where $I({\rho_t(g_1)})$
covers $I({\rho_t(g_2)})$ is a finite collection of closed intervals
and points in $[t_-, t_+]$.
\end{lemma}

\begin{proof}
Any isometric sphere is a hyperplane.  Consider the hyperboloid model
for hyperbolic space $\mathbb{H}^3$, which is the positive sheet of
$\{ v \in \RR^{3,1} : \langle v,v \rangle = -1 \}$.  In this
model, any hyperplane is of the form $\{ w \in \mathbb{H}^3: \langle
v,w \rangle = 0 \}$ for some space--like vector $w \in \RR^{3,1}$. We
may choose $w$ so that $\langle w,w \rangle = 1$.  In other words, the
norm of $w$ is 1.

Given two hyperplanes $H_1$ and $H_2$ specified by space--like vectors
$w_1$ and $w_2$ with norm 1, they are tangent if and only if $\langle
w_1, w_2 \rangle =1$.  So, consider the isometric spheres
$I(\rho_t(g_1))$ and $I(\rho_t(g_2))$, which are specified by
space--like vectors $w_1(t)$ and $w_2(t)$ with norm 1.  Then $\langle
w_1(t), w_2(t) \rangle$ is a real analytic function of $t$.  Hence,
the set of times $t$ where $I(\rho_t(g_1))$ and $I(\rho_t(g_2))$ are
tangent is finite.
\end{proof}


The next lemma essentially is a list of ways that Euclidean
hemispheres (isometric spheres) can emerge out from other Euclidean
hemispheres in a real analytic path.

\begin{lemma}\label{lemma:hemisphere-emerge}
In a real analytic path through the space of minimally parabolic,
geometrically finite uniformizations of $M$, the ways in which an
isometric sphere may become visible (or invisible) are as follows:
\begin{enumerate}
\item\label{item:bdypt} On the boundary at infinity: two nested
  isometric spheres become tangent at a point on the boundary at
  infinity, then the inner one pushes through the outer.
\item\label{item:bdyvtx} On the boundary at infinity: two visible
  isometric spheres meet at a point on the boundary at infinity, a
  third moves into the point of their intersection, then pushes
  through.
\item\label{item:Fedge} Away from the boundary at infinity: two
  visible isometric spheres meet at an edge of $\F$, a third also
  meets the length of the edge, then pushes through.
\item\label{item:Fvtx} Away from the boundary at infinity: three or
  more visible isometric spheres intersect in a vertex of $\F$,
  another moves into the vertex and then pushes through.
\end{enumerate}
It is also possible that multiple new isometric spheres become visible
or invisible simultaneously at the same points on the boundary at
infinity, or on the same edge or vertex of $\F$.

No isometric sphere may become visible without intersecting any other
visible isometric sphere.
\end{lemma}

The options for single faces becoming visible are illustrated in
Figure \ref{fig:becoming-visible}.

\begin{figure}
  \includegraphics{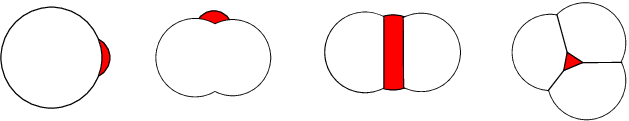}
  \caption{The ways in which isometric spheres can become visible.}
  \label{fig:becoming-visible}
\end{figure}

\begin{proof}
The fact that no isometric sphere may spontaneously arise without
intersecting any other isometric sphere follows from
Lemma \ref{lemma:iso-center-rad} and the fact that the path is real
analytic and hence continuous: each isometric sphere has positive
radius for all time.

We now show that the above four possibilities are the only
possibilities.  Suppose $I(g)$ is visible for time $t \in (t_0,
t_0+\epsilon)$, but not at time $t_0$.  Then at time $t_0$, the
isometric sphere corresponding to $I(g)$ must have one of the following
forms. 
\begin{enumerate}
\item It is covered by a single isometric sphere.  In this case, it
  will be tangent to another hemisphere at time $t_0$, then push
  through at a point that is visible on the boundary at infinity.
  This is option \eqref{item:bdypt} above.
\item It is not covered by a single isometric sphere, but is covered
  by two visible isometric spheres at time $t_0$.  Then it intersects
  two hemispheres at their edge of intersection at time $t_0$, then
  pushes through.  In this case, one of the following options holds.
  \begin{enumerate}
  \item The newly visible isometric sphere expands in such a way as to
    completely cover the old visible edge.  This gives option
    \eqref{item:Fedge} above.  
  \item The new isometric sphere slides in one direction, covering
    only a portion of the visible edge, and appearing on the boundary
    at infinity.  This gives option \eqref{item:bdyvtx} above.
  \item The new isometric sphere slides in one direction, to cover
    only a portion of the visible edge, but meets a third isometric
    sphere.  Then the new isometric sphere will become visible in a
    vertex of $\F$.  This is option \eqref{item:Fvtx} above.
  \end{enumerate}
\item Finally, at time $t_0$, if the new isometric sphere is not
  covered by either one or two isometric spheres alone, but is covered
  by three or more, then in this case the isometric sphere will meet
  the point where these isometric spheres intersect.  As it moves out
  from under the intersection, we will obtain option \eqref{item:Fvtx}
  above.
\end{enumerate}

As for multiple isometric spheres: In each case above it is possible
to have more than one hemisphere meeting the point(s) where an
isometric sphere is about to emerge.  In the case that a hemisphere is
covered by another visible hemisphere, it is possible to have multiple
hemispheres tangent at the same point, nested within each other, at
time $t_0$.  It is feasible that at time $t_0+\epsilon$, for any
sufficiently small $\epsilon>0$, a smaller hemisphere has pushed out
farther than a larger one, and so we obtain two new visible isometric
spheres.

Multiple distinct hemispheres may both meet the same edge of
intersection of visible isometric spheres, and then push through to
form new visible isometric spheres.  Similarly, multiple distinct
hemispheres may meet the point of intersection of multiple visible
isometric spheres, and push through to become visible at the same
time.  
\end{proof}

We have seen in examples that as the isometric spheres in a Ford
domain bump into each other, new isometric spheres become visible.  In
the next two lemmas, we show that this is the only way new isometric
spheres may become visible.  First, we set up some notation.

In the arguments below, we will consider a fixed collection of
isometric spheres and how they change.  Rather than considering the
entire Ford domain, we will consider instead whether given isometric
spheres are visible with respect to other isometric spheres in the
collection.  

\begin{define}
We will say an isometric sphere $I(g)$ is visible with respect to a
collection of group elements $\{k_1 \dots, k_n\} \subset \Gamma$ if
there is an open subset of $I(g)$ that is not contained in
$\Gamma_\infty(\bigcup_{j=1}^n \overline{B(k_j)})$.  Recall $B(k_j)$
is the open half space bounded by the isometric sphere $I(k_j)$.
Similarly, we say the intersection of two isometric spheres $I(g) \cap
I(h)$ is visible with respect to $\{k_1, \dots, k_n\}$ if $I(g) \cap
I(h)$ contains an open set which is not contained in
$\Gamma_\infty(\bigcup_{i=1}^n B({k_i}))$.
\label{def:visiblewrt}
\end{define}

Suppose we have a real analytic path, parameterized by time $t$,
through the interior of the space of minimally parabolic geometrically
finite uniformizations of $M$, where $M$ is a hyperbolizable
3--manifold with only one rank 2 cusp.  For any time $t$, we obtain
the region $\F(t)$ of Definition \ref{def:F}.  We may choose vertical
fundamental domains in a continuous manner to obtain a path of Ford
domains, given by finite polyhedra $P_t$.

\begin{lemma}
Suppose that at time $t_0$, the polyhedron $P_{t_0}$ is cut out by (a
vertical fundamental domain and) isometric spheres corresponding to
group elements $h_1, \dots, h_n$; and for some $\epsilon>0$, and all
time $t \in [t_0, t_0 + \epsilon)$ the combinatorics of the visible
intersections of these isometric spheres do not change.  That is, no
new visible intersections of these particular faces arise, and no
visible intersections of these faces disappear.  Then for all $t \in
[t_0, t_0 + \epsilon)$, faces corresponding to $h_1, \dots, h_n$
remain exactly those faces that are visible in a Ford spine at time
$t$.
\label{lemma:no-change-no-new}
\end{lemma}

To summarize, when the combinatorics of the visible intersections of
faces is unchanged, no new visible faces may arise.

\begin{proof}
The proof is by the Poincar{\'e} polyhedron theorem.  For any $t \in
(t_0, t_0+\epsilon)$, let $Q_t$ be the polyhedron cut out by isometric
spheres corresponding to the group elements $h_1, \dots, h_n$ and the
vertical fundamental domain of $P_t$.  Let $\mathcal{G}_t$ be the
orbit of $Q_t$ under $\Gamma_\infty$.

Because there are no new visible intersections, and no visible
intersections disappear, for each edge of $Q_t$ arising from
intersections of isometric spheres, the faces meeting that edge cycle
must be unchanged from that of $P_{t_0}$, and therefore the monodromy
around that edge is unchanged from that at time $t_0$.  Because the
monodromy is the identity at time $t_0$, it must be the identity at
time $t$, all $t \in (t_0, t_0+\epsilon)$.  Moreover, since the
dihedral angles about any edge at time $t_0$ sum to $2\pi$, and since
dihedral angles about an edge with monodromy the identity must sum to
a multiple of $2\pi$, continuity implies that the dihedral angles sum
to $2\pi$ for all $t \in (t_0, t_0+\epsilon)$.  Similarly, this is
true of translates of edges under $\Gamma_\infty$, so holds for edges
of $\mathcal{G}_t$.

Additionally, all isometric sphere faces of $\mathcal{G}_t$ are glued
isometrically by continuity: They are glued isometrically at time
$t_0$, when $\mathcal{G}_0$ is the equivariant Ford domain $\F$, and
by Lemma \ref{lemma:edge-visible} their intersections with other
isometric spheres continue to be glued isometrically.  Therefore,
visible regions continue to be glued isometrically.

By Theorem \ref{thm:poincare1}, gluing faces of $\mathcal{G}_t$ yields
a hyperbolic manifold with fundamental group generated by the face
pairings $h_1, \dots, h_n$, equivariant with respect to
$\Gamma_\infty$. Therefore when we quotient by $\Gamma_\infty$, we get
a manifold whose fundamental group is isomorphic to that of the
original manifold.  Then Lemma \ref{lemma:finding-ford} implies that
$\mathcal{G}_t$ must equal the equivariant Ford domain at time $t$.
Hence only the faces $h_1, \dots, h_n$ are visible at time $t$.
\end{proof}

\begin{lemma}
Suppose that at time $t_0$, the equivariant Ford domain $\F_{t_0}$ is
cut out by isometric spheres corresponding to group elements $h_1,
\dots, h_n$ and their translates under $\Gamma_\infty$; and for some
$\epsilon>0$ and all time $t \in [t_0, t_0+\epsilon)$, there are no
  new visible intersections of faces corresponding to the $h_j$ or
  their translates, although some visible intersections may disappear.
  Then no new visible faces arise in this time interval.
\label{lemma:nobump-nonew}
\end{lemma}

\begin{proof}
Again let $\mathcal{G}_t$ be the polyhedron cut out by isometric
spheres corresponding to $h_1, \dots, h_n$ at time $t$ and their
translates under $\Gamma_\infty$, so that $\mathcal{G}_{t_0} =
\F_{t_0}$.

If the combinatorics of intersections of isometric spheres remains as
it was at time $t_0$, then the previous lemma implies there are no new
visible faces.  So suppose the combinatorics changes.  By hypothesis,
no visible intersections of faces corresponding to $h_1, \dots, h_n$
arise.  Hence some intersection visible at time $t_0$ must disappear.
Without loss of generality, suppose faces corresponding to $h_1$ and
$h_2$ intersect visibly at time $t_0$, but not at time $t$.

If a visible edge disappears, it must do so in one of the ways of
Lemma \ref{lemma:hemisphere-emerge}.  Note that each of the ways
\eqref{item:bdypt}, \eqref{item:bdyvtx}, and \eqref{item:Fvtx} in this
lemma involve the Euclidean length of the edge shrinking to zero.
Only possibility \eqref{item:Fedge} does not.  However, in that case,
an edge disappears by sliding into another edge which was not
initially visible.  Because it was not initially visible, the two
isometric spheres meeting in this edge did not initially intersect
visibly.  Thus in case \eqref{item:Fedge}, two isometric spheres that
did not intersect visibly at time $t_0$ must intersect visibly
thereafter, contradicting hypothesis.  Therefore, this option of Lemma
\ref{lemma:hemisphere-emerge} does not happen.

Thus the Euclidean length of the visible intersection between faces
corresponding to $h_1$ and $h_2$ must decrease to zero.  Lemma
\ref{lemma:edge-visible} implies that the Euclidean length of the
image of the visible intersection under isometries corresponding to
$h_1$ and $h_2$ must also decrease to zero (as the visible edge is
mapped isometrically).  Applying the result to all edges in this edge
class, we see that the edge class must vanish from the Ford domain
entirely.  That is, all faces which meet the edge corresponding to the
visible intersection of $h_1$ and $h_2$ at time $t_0$ will cease to
intersect in pairs by time $t$ and the edge will be removed.

Now consider an edge class that remains visible with respect to faces
corresponding to $h_1, \dots, h_n$ and their translates under
$\Gamma_\infty$.  By the above argument, the edge cannot meet fewer
faces than it meets at time $t_0$, for then the entire edge would
disappear.  Since there are no additional visible intersections of the
$h_i$ and its translates, no additional face corresponding to $h_1,
\dots, h_n$ and their translates may meet the edge.  Hence a visible
edge with respect to the $h_i$ and their translates at time $t$
corresponds to a visible edge at time $t_0$, and has the same
monodromy, and therefore the monodromy is the identity.  Since this is
true for all $t \in (t_0, t_0+\epsilon)$, continuity implies the
dihedral angles about the edge sum to $2\pi$.

Next we show that faces corresponding to the $h_i$ are still glued
isometrically.  Lemma \ref{lemma:edge-visible} implies that their
intersections map to other intersections isometrically.  It could
happen that one of the faces corresponding to $h_1, \dots, h_n$ is no
longer visible with respect to the $h_i$ at time $t$.  Then we ignore
that face.  For other faces, the argument of Lemma \ref{lemma:g-ginv}
implies that if some portion of $h_j$ (or a translate) is visible with
respect to the other $h_k's$, then so must be a portion of $h_j^{-1}$.
Continuity implies visible faces glue isometrically.

By the above work, when we glue via face pairings, the result must be
a manifold by the Poincar{\'e} polyhedron theorem, Theorem
\ref{thm:poincare1}.  Because one of the faces $h_i$ may no longer be
visible, it could happen that the group generated by the pairings of
visible faces (and the quotient by $\Gamma_\infty$) no longer
generates $\pi_1(M)$, and so these isometric spheres do not give the
full equivariant Ford domain.  However, if all the $h_i$ remain
visible, then Lemma \ref{lemma:finding-ford} implies that
$\mathcal{G}_t$ is the equivariant Ford domain of our manifold, and we
are finished in this case.

So now suppose some $h_i$ becomes invisible.  In this case, there must
be some initial time at which a face $h_i$ is no longer visible, say
all the $h_i$ are visible for $t \in (t_0, t_1)$, but $h_j$ is not
visible at time $t_1$.  Up until this time, the above argument implies
that the visible isometric spheres corresponding to the $h_i$ and
their translates under $\Gamma_\infty$ cut out the equivariant Ford
domain of our manifold.

Suppose that at time $t_1$, the remaining visible isometric spheres no
longer cut out the equivariant Ford domain.  This means that at time
$t_1$, some other isometric sphere, say corresponding to $k$, must be
visible.  Lemma \ref{lemma:open} implies that there is some
$\epsilon>0$ such that the isometric sphere corresponding to $k$ is
visible for $t \in (t_1-\epsilon, t_1+\epsilon)$.  However, for $t \in
(t_1-\epsilon, t_1)$, the equivariant Ford domain is not cut out by an
isometric sphere corresponding to $k$.  This is a contradiction.

Thus in all cases, we have the setup of Lemma
\ref{lemma:finding-ford}.  So $\mathcal{G}_t$ is the equivariant Ford
domain, and hence there are no new visible isometric spheres.
\end{proof}

In a real analytic path of minimally parabolic geometrically finite
uniformizations of $M$, the dual structure to the Ford domain will be
changing.  It follows from Lemma \ref{lemma:open} that a dual edge
will exist for an open set of time.  The dual structure changes
smoothly during the path, except at a discrete set of points
corresponding to the addition or removal of a cell of the dual
structure.

In Example \ref{ex:bumping}, a new edge and a new 2--cell in the dual
structure are created when two visible isometric spheres meet across
portions of their boundaries on $\CC$.  In Example \ref{ex:sliding}, a
new edge, two new 2--cells, and a single 3--cell are created when two
visible isometric spheres slide into each other along a third visible
isometric sphere.  In this case the boundaries of the isometric
spheres on $\CC$ initially meet at a point where two other boundaries
of visible isometric spheres intersect.


\begin{define}
If in a real analytic path of minimally parabolic geometrically finite
uniformizations of $M$, two visible isometric spheres move to
intersect across portions of their boundaries on $\CC$, we will refer
to the move as \emph{bumping} at the boundary.  The reverse of this
move, where two isometric spheres pull apart at the boundary, we will
refer to as \emph{reverse bumping}.  This is the move of Example
\ref{ex:bumping}.

If an isometric sphere slides into the visible intersection of two
other isometric spheres at a point where the intersection meets the
boundary $\CC$, we call the move \emph{sliding} at the boundary.  Its
reverse we will call \emph{reverse sliding}.  This is the move of
Example \ref{ex:sliding}.

Finally, isometric spheres may also shift and change intersections
internally, without affecting the combinatorics of the boundary of the
dual structure.  We refer to these intersections as
\emph{internal moves}.
\label{def:moves}
\end{define}

For an example of an internal move, suppose two isometric spheres
$I(g)$ and $I(h)$ form a visible edge, and two additional isometric
spheres $I(k)$ and $I(\ell)$ slide together over that edge, such that
at some instant $t=t_0$ all four isometric spheres meet in a single
point.  At this instant, neither the intersection of $I(g)$ and $I(h)$
is visible, nor is the intersection of $I(k)$ and $I(\ell)$.  However,
for some $\epsilon>0$, the intersection of $I(k)$ and $I(\ell)$ will
be visible for time $(t_0, t_0+\epsilon)$, and the intersection of
$I(g)$ and $I(h)$ will be visible for time $(t_0-\epsilon, t_0)$.
This gives a ``retriangulation'' of the existing dual structure, in
which faces in the interior are removed and replaced by other faces,
and interior edges of the dual structure appear or disappear.  An
example of this phenomenon is a 2--3 Pachner move of a triangulation,
or its reverse, a 3--2 move.  See Figure \ref{fig:2-3move}.

\begin{figure}
\includegraphics{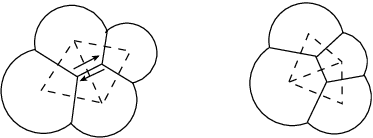}
\caption{A retriangulation of the dual structure.}
\label{fig:2-3move}
\end{figure}

\subsection{Paths and geodesic core tunnels}

We now present results that give evidence for Conjecture
\ref{conj:core-tunnel}.  We will be considering the
$(1;2)$--compression body $C$ once more.

Fix the following notation.  As before, let $\alpha$, $\beta$, and
$\gamma$ generate $\pi_1(C)$, with $\alpha$ and $\beta$ generating
$\pi_1(\D_-C) \cong (\ZZ \times \ZZ)$.  Suppose $\rho_t\co \pi_1(C)
\to PSL(2,\CC)$ is a real analytic path of minimally parabolic
geometrically finite uniformizations of $C$.  We will assume that
$\rho_t(\pi_1(\D_-C)) = \Gamma_\infty$ fixes the point at infinity of
$\HH^3$. 

The following lemma will guarantee that all structures on a particular
path through the space of minimally parabolic geometrically finite
uniformizations of $C$ have geodesic core tunnel.

\begin{lemma}
Suppose $\rho_t\co \pi_1(C) \to PSL(2,\CC)$ is a real analytic path of
minimally parabolic geometrically finite uniformizations of $C$ such
that at time $t=0$, $M_0 = \HH^3/\rho_0(\pi_1(C))$ admits a Ford spine
such that
\begin{enumerate}
\item[(a)] the isometric sphere corresponding to $\rho_0(\gamma)$ is
  visible, and
\item[(b)] the core tunnel is isotopic to the geometric dual of this
  face of the Ford spine.
\end{enumerate}
Suppose that for $t\in (0, t_0)$, the isometric sphere corresponding
to $\rho_t(\gamma)$ remains visible.  Then the core tunnel is geodesic
for all $t \in (0, t_0)$.
\label{lemma:tunnel-dual-gamma}
\end{lemma}

\begin{proof}
Consider the dual structure.  For each $t \in [0, t_0)$, since
$\rho_t(\gamma)$ is visible, there is an edge dual to it, which is a
geodesic.  The path $\rho_t$ gives a (real analytic) one--parameter
family of embedded edges dual to $\rho_t(\gamma)$.  For any $t_1 \in
(0, t_0)$, this restricts to an ambient isotopy of the edge dual to
$\rho_0(\gamma)$ to the edge dual to $\rho_{t_1}(\gamma)$.  Since the
edge dual to $\rho_0(\gamma)$ is isotopic to the core tunnel, the edge
dual to $\rho_{t_1}(\gamma)$ is also isotopic to the core tunnel, and
so the core tunnel is geodesic.  
\end{proof}

Now, we present a result that guarantees the core tunnel is geodesic
for many paths of uniformizations of $C$.  In the proof, for $g \in
\pi_1(C)$, we will sometimes denote $\rho_t(g)$ by $g_t$, or when
$\rho_t$ is clear, we will simply write $g$ to simplify notation.

\begin{theorem}
Suppose $\rho_t\co \pi_1(C) \to PSL(2,\CC)$ is a real analytic path of
minimally parabolic geometrically finite uniformizations of $C$ such
that $M_0=\HH^3/\rho_0(\pi_1(C))$ admits a Ford spine with just one
face.  Suppose for all $t>0$, there is a compression disk $D_t$
properly embedded in $C$, which does not meet any faces of the Ford
spine of $M_t = \HH^3/\rho_t(\pi_1(C))$.  Then for any $t>0$, the core
tunnel is geodesic, isotopic to an edge dual to the Ford spine.
\label{thm:disk-disjoint-ford}
\end{theorem}

\begin{figure}
\begin{center}
\begin{tabular}{ccccc}
	\includegraphics[width=1.75in]{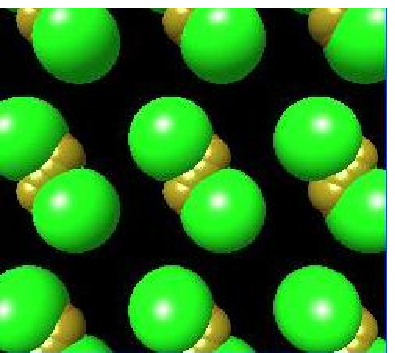} &
	\hspace{.01in} &
	\includegraphics[width=1.75in]{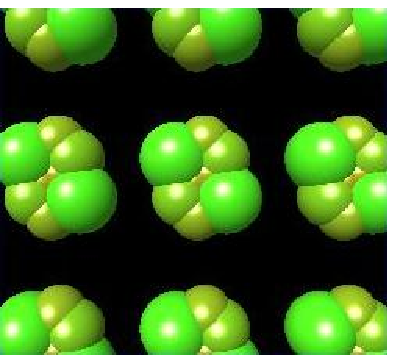} &
	\hspace{.01in} &
	\includegraphics[width=1.75in]{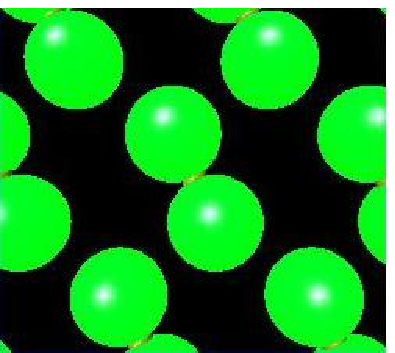}
\end{tabular}
\end{center}
\caption{Shown are examples of structures to which Theorem
\ref{thm:disk-disjoint-ford} applies.
}
\label{fig:compression-disk}
\end{figure}

\begin{proof}
Suppose the isometric sphere corresponding to $\gamma$ has remained
visible for all time $t$ in $(0, t_0)$.  We will show it is still
visible at time $t_0$.  Because isometric spheres are visible for an
open set of time, it will follow from Lemma
\ref{lemma:tunnel-dual-gamma} that the core tunnel is geodesic at time
$t_0$.

Consider a lift of the disk $D_{t_0}$ to $\HH^3$, which we will
continue to write as $D_{t_0}$, abusing notation slightly.

Without loss of generality, we may assume $\D D_{t_0}$ encircles a
single connected component of the isometric spheres of $\F$, for if
not, we may replace $D_{t_0}$ with a disk which has this property, as
follows.  If $\D D_{t_0}$ encircles more than one connected component,
then there is an arc $\alpha$ in $\CC$ from $\D D_{t_0}$ to itself
which meets no isometric spheres of $\F$.  Then there is a disk in
$\HH^3$ with boundary on $\alpha \subset \CC$ and on $D_{t_0}$ which
is disjoint from the isometric spheres of $\F$ and with interior
disjoint from $D_{t_0}$.  Replace $D_{t_0}$ with a portion of
$D_{t_0}$ and this new disk with boundary $\alpha$, reducing the
number of components encircled by $D_{t_0}$.  Repeat, as necessary, to
obtain $D_{t_0}$ whose boundary encircles a single connected component
of the isometric spheres of $\F$.

Without loss of generality, we may assume $\D D_{t_0}$ encircles
$I({\gamma})$ at time $t_0$.  Then note that $I({\gamma})$ cannot meet
$p(I(\gamma))$ for any $p \in \Gamma_\infty \setminus \{1\}$, or else
the faces $p^n(I(\gamma))$, $n\in\ZZ$ would form an infinite strip of
isometric spheres, and $\D D_{t_0}$ would have to intersect this
strip, contradicting assumption.  So we may assume $I(\gamma)$ (and
hence $I(\gamma^{-1})$) meets none of its translates under
$\Gamma_\infty = \Gamma_\infty(t_0)$.

Change generators, if necessary, so that the isometric sphere
$I(\gamma)$ is at least as close to $I(\gamma^{-1})$ as to any of the
translates of $I(\gamma^{-1})$ under $\rho_{t_0}(\Gamma_\infty)$ at
time $t_0$.

Suppose first that $I({\gamma})$ and $I({\gamma^{-1}})$ are disjoint
(or only meet at a single point on the boundary at infinity).  Then in
this case, as in the proof of Lemma \ref{lemma:simple-ford}, the
Poincar{\'e} polyhedron theorem implies that the object obtained by
gluing isometric spheres corresponding only to $I(\gamma)$ and
$I(\gamma^{-1})$ and their translates under $\Gamma_\infty$,
quotiented out by $\Gamma_\infty$, must be a manifold with fundamental
group isomorphic to $\pi_1(C)$.  Then Lemma \ref{lemma:finding-ford}
implies that the equivariant Ford domain in this case consists only of
faces $I(\gamma)$ and $I(\gamma^{-1})$ (and their translates under
$\Gamma_\infty$).  Thus $M_{t_0}$ must have a simple Ford spine
consisting of one face, so by Proposition \ref{prop:simple-ford}, the
core tunnel is geodesic.

Next suppose $I({\gamma})$ and $I({\gamma^{-1}})$ intersect.  Then
they (i.e.\ their boundaries) are contained within the region of $\CC$
bounded by $\D D_{t_0}$.  Let $I(g)$ and $I(h)$ be any isometric
spheres within this region.  Then note that for any nontrivial
parabolic $p\in \Gamma_\infty \setminus \{1\}$, $p (I(g))$ cannot meet
$I(h)$, for $p(I(g))$ must lie outside the region bounded by $\D
D_{t_0}$.

We claim that in this case, all visible isometric spheres in the
region bounded by $\D D_{t_0}$ are of the form $I(g)$ for $g$ an
element of the cyclic group $\langle \gamma \rangle$.  Again this will
follow from Lemma \ref{lemma:finding-ford}, as follows.  Consider the
isometric spheres corresponding to the cyclic group $\langle \gamma
\rangle$.  Ford domains of cyclic groups have been studied by
J{\o}rgensen \cite{jorgensen} and Drumm and Poritz
\cite{drumm-poritz}.  In particular, it is known that $\langle \gamma
\rangle$ is geometrically finite, so a finite number of isometric
spheres corresponding to this group will be visible with respect to
the other isometric spheres of the group.  Moreover, they will glue to
give a manifold, namely a layered solid torus.  Additionally, the Ford
domain for $\langle \gamma \rangle$ is connected.  Hence it lies
entirely within $\bdy D_{t_0}$, and thus it is disjoint from all its
translates under $\Gamma_\infty$.  Therefore when we consider all
translates under $\Gamma_\infty$ of visible isometric spheres
corresponding to the cyclic group $\langle \gamma \rangle$, the result
is a domain in $\HH^3$ cut out by isometric spheres, which glue to
give a manifold.  If we further take the quotient by $\Gamma_\infty$
then we obtain a manifold homeomorphic to the $(1;2)$--compression
body.  The fundamental group of this quotient manifold clearly
contains $\Gamma_\infty$; it also contains $\gamma$ because it
contains all of $\langle \gamma \rangle$.  Hence the fundamental group
of this manifold is $\rho_t(\pi_1(C))$.  Lemma
\ref{lemma:finding-ford} implies that we have found the entire
(equivariant) Ford domain.

Work of J{\o}rgensen \cite{jorgensen:cyclic} and Drumm and Poritz
\cite{drumm-poritz} implies that the face $I(\gamma)$ is visible in
the Ford domain of $\langle \gamma \rangle$.  Therefore in our case,
$I(\gamma)$ must remain visible at time $t=t_0$ (this is contained in
\cite[Theorem 7.9]{drumm-poritz}, see also the two paragraphs before
the statement of that theorem).  Then our result follows from Lemma
\ref{lemma:tunnel-dual-gamma}.
\end{proof}

By Lemma \ref{lemma:tunnel-dual-gamma}, in a real analytic path of
minimally parabolic geometrically finite uniformizations of $C$ which
begins with a simple Ford spine, if the isometric spheres corresponding
to $\gamma$ and $\gamma^{-1}$ remain visible throughout, then the core
tunnel remains visible.  We found no topological obstruction to the
isometric sphere of $\gamma$ being covered.  However, in practice, we
were unable to find examples of paths in which this occurred.  All
such examples led to indiscrete groups.

Figure \ref{fig:no-core} shows examples of Ford domains obtained by
our computer program which are not guaranteed to have a geodesic core
tunnel by Theorem \ref{thm:disk-disjoint-ford}.  However, each of
these can be shown to have geodesic core tunnel by observation.  In
particular, the face $I({\gamma})$ is visible always for each of these
examples.  Thus by Lemma \ref{lemma:tunnel-dual-gamma}, the core
tunnel is geodesic for each of these structures.  Moreover, it is
actually dual to a face of the Ford spine.

\begin{figure}
\begin{center}
\begin{tabular}{ccccc}
  \includegraphics[width=1.75in]{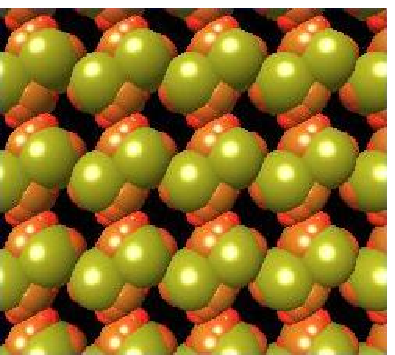}
  & \hspace{0.01in} &
	\includegraphics[width=1.75in]{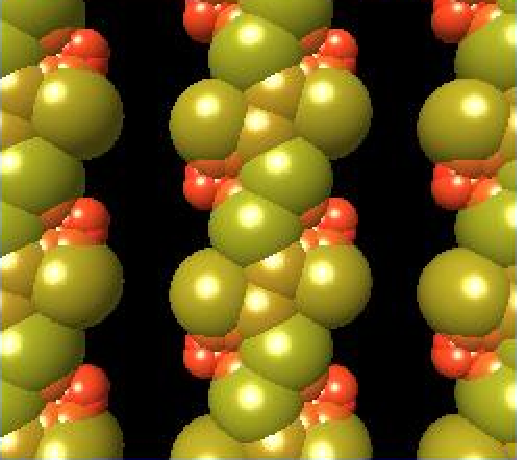} &
	\hspace{.01in} & 
	\includegraphics[width=1.75in]{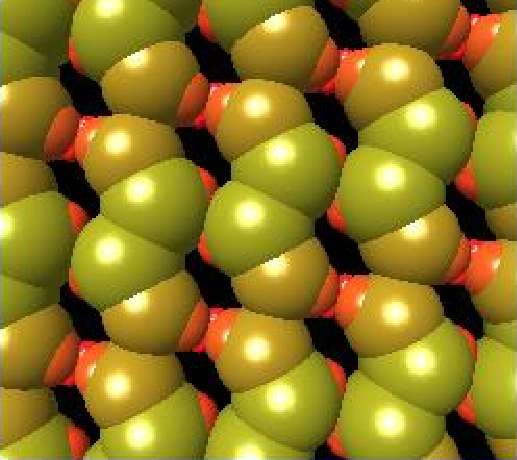}
\end{tabular}
\end{center}
\caption{Snapshots along paths.  In the figures shown, the core tunnel
  is geodesic because $I(\gamma)$ remains visible.}
\label{fig:no-core}
\end{figure}

This leads us to the following strengthening of Conjecture
\ref{conj:core-tunnel}.

\begin{conjecture}
In any geometrically finite hyperbolic structure on a
$(1;2)$--compression body, the core tunnel is isotopic to a geodesic
dual to a face of the Ford domain.
\label{conj:dual-ford}
\end{conjecture}

The analogue of Conjecture \ref{conj:dual-ford} is false for finite
volume manifolds.  Sakuma and Weeks conjectured in \cite{sakuma-weeks}
that core tunnels in knot complements were isotopic to edges of the
Epstein--Penner canonical polyhedral decomposition of the knot
complement \cite{epstein-penner}, which is dual to the Ford domain.
Sakuma and Weeks' conjecture was shown to be false by Heath and Song
in \cite{heath-song}.  They showed that the knot $P(-2,3,7)$ has four
distinct core tunnels, but only three edges of the canonical
polyhedral decomposition.

Nevertheless, we believe Conjecture \ref{conj:dual-ford} will be true
for compression bodies.

\subsection{Relation to Computer Procedure}

Conjecture \ref{conj:dual-ford} is intricately related to Procedure
\ref{procedure}.  

Suppose Conjecture \ref{conj:dual-ford} is false.  Then for some
geometrically finite hyperbolic structure, the faces corresponding to
$\gamma$, $\gamma^{-1}$ are not visible.  In this case, it is not
clear whether Procedure \ref{procedure} will actually find and draw
all the faces of the Ford domain.  Additionally, since at least
initially the procedure is drawing isometric spheres that will not be
faces of the Ford domain, it no longer follows that the procedure is
drawing faces of a convex fundamental domain for the group $\Gamma$.
Hence work of Bowditch will not apply, and the procedure may never
terminate.

Similarly, suppose a face $F$ of the Ford domain initially arose as
the intersection of two visible faces in a path through Ford domains,
but that later in the path, those visible faces or their edge of
intersection becomes covered by some other face.  Then it could
possibly be the case that Procedure \ref{procedure} never draws face
$F$.  However, again based on experimental evidence, this seems unlikely.  

How might a face become invisible?  If it is covered by other faces.
In the interior, such a move would occur as the dual of a 3--2 move of
tetrahedra, or some similar move.

\begin{question}
For any geometrically finite hyperbolic structure on a
$(1;2)$--compression body, does there exist a smooth path of Ford
domains from the simple structure to this particular structure for
which there are no internal moves on the Ford domain?  
\label{quest:no-intern}
\end{question}

Note that an affirmative answer to Question \ref{quest:no-intern}
would imply Conjecture \ref{conj:dual-ford}, as the faces
corresponding to $\gamma$, $\gamma^{-1}$ cannot disappear as other
faces slide together and apart, meeting only on the boundary $\CC \cap \F$.  

There is some evidence for a positive answer to Question
\ref{quest:no-intern}.  Our proof of Theorem
\ref{thm:disk-disjoint-ford} shows that 2--3 moves are impossible in
the core of the Ford domain when the Ford domain has the form of that
theorem.  Interestingly, in the case of once--punctured tori,
J{\o}rgensen also found that internal moves in the Ford domain are
impossible \cite{jorgensen}.  However, his proof was similar to our
proof of Theorem \ref{thm:disk-disjoint-ford}, and will not answer
Question \ref{quest:no-intern}.

Using our computer program, we have found that 2--3 moves do occur in
the $(1;2)$--compression body setting.  However, in all examples
encountered, there was a straightforward way to reparameterize the
path of structures to avoid these moves.

\bibliographystyle{hamsplain}
\bibliography{biblio}

\end{document}